\documentclass[letterpaper,10pt]{amsart}

\usepackage[all]{xy}                        %

\CompileMatrices                            

\UseTips                                    

\input xypic
\usepackage[bookmarks=true]{hyperref}       

\usepackage{amssymb,latexsym,amsmath,amscd}
\usepackage{xspace}
\usepackage{color}

\usepackage{graphicx}

\reversemarginpar

\theoremstyle{plain}
\newtheorem{theorem}{Theorem}[section]
\newtheorem*{theorem*}{Theorem}
\newtheorem{proposition}[theorem]{Proposition}
\newtheorem{corollary}[theorem]{Corollary}
\newtheorem{lemma}[theorem]{Lemma}

\theoremstyle{definition}
\newtheorem{definition}[theorem]{Definition}

\newtheorem{notation}[theorem]{Notation}
\newtheorem{possibility}[theorem]{Possible Cases}

\newtheorem{remark}[theorem]{Remark}
\newtheorem{example}[theorem]{Example}

\newcommand{\enm}[1]{\ensuremath{#1}}          %
\newcommand{\op}[1]{\operatorname{#1}}
\newcommand{\cal}[1]{\mathcal{#1}}

\newcommand{\CC}{\enm{\mathbb{C}}}

\newcommand{\NN}{\enm{\mathbb{N}}}

\newcommand{\QQ}{\enm{\mathbb{Q}}}
\newcommand{\ZZ}{\enm{\mathbb{Z}}}

\newcommand{\PP}{\enm{\mathbb{P}}}

\newcommand{\Aa}{\enm{\cal{A}}}
\newcommand{\Bb}{\enm{\cal{B}}}

\newcommand{\Ee}{\enm{\cal{E}}}
\newcommand{\Ff}{\enm{\cal{F}}}
\newcommand{\Gg}{\enm{\cal{G}}}

\newcommand{\Ii}{\enm{\cal{I}}}

\newcommand{\Ll}{\enm{\cal{L}}}

\newcommand{\Oo}{\enm{\cal{O}}}

\newcommand{\Ss}{\enm{\cal{S}}}

\renewcommand{\phi}{\varphi}
\renewcommand{\theta}{\vartheta}
\renewcommand{\epsilon}{\varepsilon}

\newcommand{\Hom}{\op{Hom}}
\newcommand{\Ext}{\op{Ext}}

\newcommand{\End}{\op{End}}
\newcommand{\Aut}{\op{Aut}}

\newcommand{\supp}{\op{Supp}}

\newcommand{\Image}{\op{Im}}

\renewcommand{\to}[1][]{\xrightarrow{\ #1\ }}

\newcommand{\old}[1]{}

\begin{document}

\title[aCM sheaves on the double plane]{aCM sheaves on the double plane}

\author{E. Ballico, S. Huh, F. Malaspina and J. Pons-Llopis}

\address{Universit\`a di Trento, 38123 Povo (TN), Italy}
\email{edoardo.ballico@unitn.it}

\address{Sungkyunkwan University, Suwon 440-746, Korea}
\email{sukmoonh@skku.edu}

\address{Politecnico di Torino, Corso Duca degli Abruzzi 24, 10129 Torino, Italy}
\email{francesco.malaspina@polito.it}

\address{Department of Mathematics, Kyoto University, Kyoto, Japan}
\email{ponsllopis@gmail.com}

\thanks{The first and third authors are partially supported by GNSAGA of INDAM (Italy) and MIUR PRIN 2015 \lq Geometria delle variet\`a algebriche\rq. The second author is supported by Basic Science Research Program 2015-037157 through NRF funded by MEST and the National Research Foundation of Korea(KRF) 2016R1A5A1008055 grant funded by the Korea government(MSIP). The forth author is supported by a FY2015 JSPS Postdoctoral Fellowship. The third and forth authors thank Sungkyunkwan University for the warm hospitality.}



\begin{abstract}
The goal of this paper is to start a study of aCM and Ulrich sheaves on non-integral projective varieties. We show that any aCM vector bundle of rank two  on the double plane is a direct sum of line bundles. As a by-product, any aCM vector bundle of rank two on a sufficiently high dimensional quadric hypersurface also splits. We consider aCM and Ulrich vector bundles on a multiple hyperplanes and prove the existence of such bundles that do not split, if the multiple hyperplane is linearly embedded into a sufficiently high dimensional projective space. Then we restrict our attention to the double plane and give a classification of aCM sheaves of rank at most $3/2$ on the double plane and describe the family of isomorphism classes of them.
\end{abstract}

\maketitle


\section{Introduction}
Ever since Horrocks proved that a vector on the projective space splits as the sum of line bundles if and only if it has no intermediate cohomology, there have been two directions of study: one is to find out criterion of coherent sheaves that do not split on a given projective scheme $X\subset\PP^n$, that is, the equivalent condition with which a coherent sheaf is a direct sum of line bundles $\Oo_X(t)$, and the other is to classify indecomposable coherent sheaves that have no intermediate cohomology on $X$, i.e. $H^i(X, \Ee(t))=0$ for any $t\in \ZZ$ and $i=1,\ldots, \dim X-1$; they are called {\it arithmetically Cohen-Macaulay} (for short, aCM). About the former direction, the case of hyperquadrics was studied in \cite{ottav}. Madonna in \cite{Mad2} and Kumar, Rao and Ravindra in \cite{krr1,krr2} focused on criteria for vector bundles that do not split, usually of low rank, on hypersurfaces of higher degree. About the latter direction, the classification of aCM vector bundles has been done for several projective varieties such as smooth quadric hypersurfaces in \cite{Kapranov, Knorrer}, cubic surfaces in \cite{CH, faenzi}, prime Fano threefolds in \cite{Madonna}, Grassmannian varieties in \cite{CM1} and others. In fact, in \cite{eisenbud-herzog:CM} a complete list of varieties supporting a finite number of aCM sheaves is provided. Varieties that only support one dimensional families of aCM vector bundles, \emph{tame varieties}, are known by the classical work of Atiyah in \cite{atiyah:elliptic} for elliptic curves, and much more recently by work of Faenzi and Malaspina in \cite{faenzi-malaspina} for rational scrolls of degree four. In \cite{CMP} it is shown that all the Segre varieties have a \emph{wild} behaviour, namely they support families of arbitrary dimension of aCM sheaves. Finally, it has been shown that the rest of aCM integral projective varieties which are not cones are wild; see \cite{faenzi-pons}.

Along these lines, the particular class of aCM sheaves supporting the maximum permitted number of global sections has raised the attention of many algebraic geometers in the last years. They are the so-called Ulrich sheaves; see \cite{ESW}. The existence of an Ulrich sheaf on a projective integral variety $X\subset\PP^n$ has very strong consequences. For instance, this implies that the cone of cohomology tables of vector bundles on $X$ are the same as the one on a projective space of the same dimension; see \cite{ES}. Moreover, the Cayley-Chow form of $X$ has a particular nice description; see \cite{ESW}. Eisenbud and Schreyer stated the existence of Ulrich sheaves on any projective schemes as a problem in \cite[page 543]{ESW}. Until now, this problem has been solved for arbitrary curves, for some minimal smooth surfaces of Kodaira dimension less than or equal to zero and for some sporadic cases of higher dimension; see \cite{CM}. A nice up-to-date account can be found in \cite{beauville}.

As it can be seen from the previous paragraphs, up to now most of the research on aCM and Ulrich sheaves has been restricted to the case of integral varieties. The main goal of this paper is to start the study of the aforementioned issues for non-integral ones. For instance, we expect, relying on the results from this paper, that original and interesting behaviours for aCM sheaves on non-reduced varieties can be revealed. In this paper we work on the classification of aCM sheaves on the double plane $X$, i.e. the projective plane $H\subset \PP^3$ with multiplicity two over an algebraically closed field of characteristic $0$. It is known from \cite[Theorem A]{B} that any aCM sheaf $\Ee$ on $X$ admits a $\Oo_{\PP3}$-free resolution of length one. Our first result is on the aCM vector bundles of rank two on $X$.

\begin{theorem}\label{pprop}
Every aCM vector bundle of rank two on $X$ is a direct sum of two line bundles.
\end{theorem}

\noindent Theorem \ref{pprop} is not extended to higher rank; indeed, we find a family of indecomposable aCM vector bundle of rank four on $X$; see Proposition \ref{errb1}. On the other hand, it is extended to higher dimensional quadric hypersurfaces, using an inductive argument; see Lemma \ref{cor3.9bis} and Corollary \ref{==1}. There is an indecomposable aCM vector bundle of rank two on any union of planes with multiplicities if at least one plane occurs with multiplicity one; see Corollary \ref{poppp}. On the other hand, see Proposition \ref{g6} and Remark \ref{g8} for a conditional existence theorem of aCM vector bundles on any configuration of hyperplanes with multiplicities.

Then we focus our attention to general aCM sheaves on $X$ of rank at most $3/2$. In general the rank need not to be integer (see Definition \ref{def}). Our technical ingredient is to twist a given aCM sheaf $\Ee$ by $\Oo_X(t)$ with $t\in \ZZ$ so that $\Ee$ is $0$-regular, but not $(-1)$-regular (we call it \emph{minimally regular}). It guarantees the existence of a non-zero map $u:\Ee \rightarrow \Ii_A(1)$ for a closed subscheme $A$ that is cut out scheme-theoretically in $X$ by linear equations. Thus we have candidate for a possible subscheme $A$ and describe $\ker (u)$ in each case. The structure sheaf $\Oo_H$ of $H$ can be shown to be the unique aCM sheaf of rank $1/2$ up to twist. For higher rank we introduce the notion of a {\it layered} sheaf which admits a filtration whose successive quotient is isomorphic to $\Oo_H$ up to twist. It turns out that every aCM sheaf of rank one on $X$ is layered.

\begin{theorem}\label{aprop2}
If $\Ee$ is an aCM sheaf of rank one on $X$, then it is isomorphic to
\begin{enumerate}
\item [(i)] a line bundle on $X$
\item [(ii)] a direct sum of two aCM sheaves of rank $1/2$, or
\item [(iii)] there is $t\in \ZZ$ such that $\Ee (t)$ is isomorphic to the ideal sheaf of a plane curve in $H$.
\end{enumerate}
\end{theorem}

In particular, from the previous theorem, we can spot a new kind of wildness presumably that does not occur for smooth projective varieties; compare to \cite{faenzi-pons}:

\begin{proposition}\label{verywild}
The double plane $X$ is of wild type in a very strong sense, that is, there exist arbitrarily large dimensional families of pairwise non-isomorphic aCM sheaves of fixed rank one on $X$.
\end{proposition}

Observing from Theorem \ref{aprop2} that every aCM sheaf of rank one on $X$ is layered in a sense that each aCM sheaf admits a filtration whose successive quotients are aCM sheaves of rank $1/2$. For aCM sheaves of rank $3/2$ on $X$ the situation is richer:

\begin{theorem}\label{iii}
Let $X\subset\PP^3$ be the double plane.
\begin{enumerate}
\item There exists a non-layered Ulrich stable sheaf $\overline{\Ee}$ of rank $3/2$ on $X$. Moreover, every non-layered aCM sheaf of rank $3/2$ on $X$ is isomorphic to $\overline{\Ee}(t)$ for some $t\in \ZZ$.
\item For any layered aCM  sheaf $\Ee$ on $X$ of rank $3/2$,  there exists an integer $t\in \ZZ$ such that either
\begin{itemize}
\item [(i)] $\Ee (t)$ admits a filtration $0=\Ee_0 \subset \Ee_1 \subset \Ee_2 \subset \Ee_3=\Ee (t)$ such that $\Ee_i/\Ee_{i-1} \cong \Oo_H$ for each $i=1,2,3$;
\item [(ii)] it fits into the following sequence with $a\ge d$,
$$0\rightarrow \Ii_C(a) \rightarrow \Ee (t) \rightarrow \Oo_H \rightarrow 0,$$
\noindent for $C\subset H$ a plane curve of degree $d$;
\item [(iii)] it fits into the following sequence with $0\le b<d$,
$$0\rightarrow \Oo_H(b) \rightarrow \Ee (t) \rightarrow \Ii_C(d)\rightarrow 0,$$
\noindent for $C\subset H$ a plane curve of degree $d$.
\end{itemize}
\end{enumerate}
\end{theorem}
\noindent The only non-layered sheaf $\overline{\Ee}$ in (1) of Theorem \ref{iii} is a non-trivial extension of $\Ii_p(1)$ by $\Oo_H(-1)$ as $\Oo_X$-sheaves, where $\Ii_p(1)$ is the ideal sheaf of a point $p\in X$. It turns out that the sheaf $\overline{\Ee}$ is independent of the choice of the point $p\in X$ up to twists; we refer to Propositions \ref{aprop3} and \ref{aprop4}.  For the description of type (2-i) we refer to Lemma \ref{aaa1} with Remark \ref{aaaa1}. In case $a=d=1$ of type (2-i) we get Ulrich sheaves. By Lemma \ref{iid} any sheaf fitting into the non-trivial sequence in (2-ii) is indecomposable. In fact, the isomorphism classes of such sheaves with $a>\deg (C)$ are parametrized by the orbits of $\mathrm{Aut}(\Ii_C(a))$ acting on $\Ext_X^1(\Oo_H, \Ii_C(a))\setminus \{0\}$. Lastly the description of type (2-iii) may be seen in Example \ref{bbb11} and \ref{bbb1}.

Then we describe the (non)-existence of (non)-layered Ulrich sheaves on $X$. It turns out that there is no layered Ulrich vector bundle on $X$, while there exist some non-layered Ulrich vector bundles of rank divisible by four; see Propositions \ref{ff2.1-} and \ref{u1}. Indeed, there exists a layered indecomposable Ulrich sheaf with arbitrary half integral rank; see Theorem \ref{tthhmm}.

Let us summarize here the structure of this paper. In section $2$ we introduce the definition of aCM sheaves and a generalized notion of rank, possibly not an integer. In section $3$ we collect several technical results on the restriction of aCM vector bundle and show the existence of an aCM vector bundle that does not split of arbitrary rank on an arbitrary generalized hyperplane arrangement, when it is embedded linearly into a sufficiently high dimensional projective space. In section $4$ we show that every aCM vector bundle of rank two on $X$ splits. As a generalization we also show that any aCM vector bundle of rank two splits on any sufficiently high dimensional quadric hypersurface. In section $5$ we deal with the aCM sheaves of rank $1/2$ and $1$ to give their complete classification, which induces the wildness of the double plane. We also show the existence of arbitrarily large dimensional family of indecomposable layered aCM sheaves of any rank at least one, which also implies the wildness. In section $6$ we focus our attention to the case of rank $3/2$. We start from calculating numeric data of extension groups on aCM sheaves of lower ranks. Our main result in this section is the unique existence of non-layered aCM sheaf of rank $3/2$ up to a twist, which is also semistable and simple. Then we describe the family of isomorphism classes of non-layered sheaves. Finally in section $7$, we prove the existence of layered indecomposable Ulrich sheaves for each half integral rank.

We are deeply grateful to the anonymous referee for numerous corrections and very stimulating observations.


\section{Preliminaries}
Throughout the article our base field $\mathbf{k}$ is algebraically closed of characteristic $0$. We always assume that our projective schemes $X\subset \PP^N$ have pure dimension at least two and are arithmetically Cohen-Macaulay, namely, $h^1(\Ii _X(t)) =0$ for all $t\in \ZZ$ and $h^i(\Oo _X(t)) =0$ for all $t\in \ZZ$ and all $i=1,\dots ,\dim X-1$. Then by \cite[Th\'eor\`eme 1 in page 268]{serre} all local rings $\Oo _{X,x}$ are Cohen-Macaulay of dimension $\dim X$. From $h^1(\Ii _X)=0$ we see that $X_{\mathrm{red}}$ is connected. Since in all our results we have $N =\dim X+1$, the reader may just assume that $X$ is a hypersurface. For a vector bundle $\Ee$ of rank $r\in \ZZ$ on $X$, we say that $\Ee$ {\it splits} if $\Ee \cong \oplus_{i=1}^r \Oo_X(t_i)$ for some $t_i\in \ZZ$ with $i=1,\ldots, r$.

We always fix the embedding $X\subset \PP^N$ and the associated polarization $\Oo_X(1)$. For a coherent sheaf $\Ee$ on a closed subscheme $X$ of a fixed projective space, we denote $\Ee \otimes \Oo_X(t)$ by $\Ee(t)$ for $t\in \ZZ$. For another coherent sheaf $\Gg$, we denote by $\mathrm{hom}_X(\Ff, \Gg)$ the dimension of $\Hom_X(\Ff, \Gg)$, and by $\mathrm{ext}_X^i(\Ff, \Gg)$ the dimension of $\mathrm{Ext}_X^i (\Ff, \Gg)$.

Now recall that the {\it depth} of a module $M$ over a local ring $A$ is defined to be the maximal length of $M$-regular sequence; see \cite[page 4]{hl}. We say that a coherent sheaf $\Ee$ on $X$ has pure depth $k$, if the depth of $\Ee_x$ over $\Oo_{X,x}$ is $k$ for all $x\in X$. We denote the pure depth by $\mathrm{depth}(\Ee)$. For a full account on depth and related properties, see \cite[Chapter 1]{BH}.

\begin{definition}\label{deff}
A coherent sheaf $\Ee$ on $X\subset\PP^N$ is called {\it arithmetically Cohen-Macaulay} (for short, aCM) if the following hold:
\begin{itemize}
\item [(i)] the dimension of the support of $\Ee$ is equal to $\dim (X)$,
\item [(ii)] the stalk $\Ee_x$ has positive depth for any point $x$ on $X$, and
\item [(iii)] $H^i(\Ee(t))=0$ for all $t\in \ZZ$ and $i=1, \ldots, \dim (X)-1$.
\end{itemize}
\end{definition}

\begin{remark}\label{rremm}
\begin{itemize}
\item [(i)] Since $\Ee$ is coherent, the condition in Definition \ref{deff} that the stalk $\Ee_x$ has positive depth for each $x\in X$ is equivalent to $H^0(\Ee(-t))=0$ for $t\gg 0$ by \cite[Th\'eor\`eme 1 in page 268]{serre}. Furthermore, if $H^1(\Ee(-t))=0$ for all $t\gg 0$, then every stalk of $\Ee$ has depth at least two also by \cite[Th\'eor\`eme 1 in page 268]{serre}. This implies that it has depth $\dim (X)$, namely it is locally Cohen-Macaulay, again by \cite[Th\'eor\`eme 1 in page 268]{serre} together with the vanishing of the intermediate cohomologies of $\Ee$.
\item [(ii)] Notice that being aCM does not depend on a twist of $\Ee$ by $\Oo_X(1)$.
\end{itemize}
\end{remark}

If $\Ee\not\cong 0$ is a coherent sheaf on a closed subscheme $X$ of a fixed projective space, then we may consider its Hilbert polynomial $\mathrm{P}_{\Ee}(m)\in \QQ [m]$ with the leading coefficient $\mu(\Ee)/d!$, where $d$ is the dimension of $\mathrm{Supp}(\Ee)$ and $\mu=\mu(\Ee)$ is called the {\it multiplicity} of $\Ee$. The {\it normalized} Hilbert polynomial of $\Ee$ is defined to be the Hilbert polynomial of $\Ee$ divided by $\mu (\Ee)$.

\begin{definition}\label{def}
If $\dim \mathrm{Supp}(\Ee)=\dim (X)$, then the {\it rank} of $\Ee$ is defined to be
$$\mathrm{rank}(\Ee)=\frac{\mu(\Ee)}{\mu(\Oo_X)}.$$
Otherwise it is defined to be zero.
\end{definition}
For an integral scheme $X$, the rank of $\Ee$ is the dimension of the stalk $\Ee_x$ at the generic point $x\in X$. But in general $\mathrm{rank}(\Ee)$ needs not be integer.

\begin{definition}
An {\it initialized} coherent sheaf $\Ee$ on $X\subset\PP^N$ (i.e. $0=h^0(\Ee(-1))< h^0(\Ee)$) is called an {\it Ulrich sheaf} if it is aCM and $h^0(\Ee)=\deg (X)\mathrm{rank}(\Ee)$.
\end{definition}

Ulrich sheaves have received a lot of attention during the last years. It is a central problem on this area to know which (if all) projective schemes support Ulrich sheaves. We refer the reader to \cite{beauville} and \cite{ESW} for a complete introduction to the theory of Ulrich sheaves.

The following definition will be used extensively throughout the paper:

\begin{definition}
A coherent sheaf $\Ee$ of positive depth on $X\subset\PP^N$ is called \emph{minimally regular} if it is $0$-regular but it is not $(-1)$-regular.
\end{definition}

For any coherent sheaf $\Ee$ of positive depth there exists $t\in\ZZ$ such that $\Ee(t)$ becomes minimally regular. Let us recall that a minimally regular sheaf is globally generated. Notice also that any Ulrich sheaf is minimally regular; see \cite{ESW}.

Let $S=\mathbf{k}[x_0, \ldots, x_n]$ and $f\in S$ be a nonzero homogeneous element in the irrelevant ideal. Then an aCM sheaf $\Ee$ on the hypersurface $X=V(f)$ is given as the cokernel of the map $\mathrm{M}$:
\begin{equation}\label{mf}
0\to \oplus_{i=1}^e\Oo_{\PP^n}(a_i) \stackrel{\mathrm{M}}{\to} \oplus_{i=1}^e\Oo_{\PP^n}(b_i)\to \Ee \to 0,
\end{equation}
where $\mathrm{M}=(m_{ij})$ is a square matrix of order $e$ with homogeneous entries of degree $\mathrm{max}\{b_i-a_j, 0\}$ in $S$; see \cite[Theorem A]{B}. In particular we get $\det (\mathrm{M})=f^{a}$ with $a=\sum_{i=1}^e (b_i-a_i)/\deg (f)$. If $X$ is irreducible, then we get $a=\mathrm{rank}(\Ee)$; see \cite[Proposition 5.6]{E}.

Now we pay our attention to a special type of schemes, an arbitrary finite union of hyperplanes of $\PP^{n+1}$ with prescribed multiplicities: fix $k$ positive integers $m_1, \ldots, m_k$ and $k$ distinct hyperplanes $M_1, \ldots, M_k$ of $\PP^{n+1}$ such that $M_1\cup \cdots \cup M_k$ is not necessarily a normal crossing divisor. Set $m:=m_1+\dots + m_k$ and
$$X=X_n[m_1M_1, \dots, m_kM_k]:=m_1M_1\cup \dots \cup m_kM_k$$
as a hypersurface of degree $m$ in $\PP^{n+1}$. Thus it is a polarized projective scheme with $\Oo_X(1)$ as its polarization.

\begin{lemma}\label{bo2}
For $X= X_n[m_1M_1, \dots ,m_kM_k]$, we have $\mathrm{Pic}(X) \cong \ZZ\langle \Oo _X(1)\rangle$.
\end{lemma}

\begin{proof}
Set $m:= m_1+\cdots +m_k$. Since $\mathrm{Pic}(\PP^n) \cong \ZZ \langle \Oo _{\PP^n}(1)\rangle$, we may assume $m\ge 2$ and use induction on $m$, i.e. we assume that the lemma is true for smaller multiplicities. For a fixed $\Ll \in \mathrm{Pic}(X)$, it is sufficient to prove that if $\Ll _{|M_k} \cong \Oo
_{M_k}$, then $\Ll \cong \Oo _X$. Set $Y:= X_n[m_1M_1, \dots ,(m_k-1)M_k]$, with the convention that $M_k$ does not appear inside the square brackets if $m_k=1$. Since $m\ge 2$ and $\Ll _{|M_k} \cong \Oo _{M_k}$, by tensoring the exact sequence
$$0\to \Oo _{M_k}(1-m) \to \Oo _X\to \Oo _Y\to 0$$
with $\Ll$, we get that the restriction map $H^0(\Ll )\rightarrow H^0(\Ll _{|Y})$ is bijective.

Assume for the moment $m_k\ge 2$, i.e. $M_k\subseteq Y$. By the inductive assumption we have $\Ll _{|Y}\cong \Oo _Y$. Thus we get $h^0(\Ll )=1$ and a section $\sigma \in H^0(\Ll)$ with no zero at every point of $Y_{\mathrm{red}} =X_{\mathrm{red}}$. Thus $\sigma:\Oo _X\rightarrow \Ll$ is an isomorphism.

Now assume $m_k=1$. Exchanging the labels of the planes we see that it is sufficient to prove the assertion in the case $m_i=1$ for all $i$, in which we have $m=k$. Fix $i\in \{1,\dots ,k-1\}$ and set $L_i:= M_i\cap M_k$. Since we have $\Ll _{|M_k}\cong \Oo _{M_k}$, we get $\Ll _{|L_i} \cong \Oo _{L_i}$, in particular $\Ll _{|M_i}\cong \Oo _{M_i}$ for all $i$. The inductive assumption on $m$ gives $\Ll _{|Y} \cong \Oo _Y$ and we may repeat the argument given for the case $m_k\ge 2$.
\end{proof}

In particular, there is no ambiguity on the choice of the ample generator $\Oo_X(1)$ of $\mathrm{Pic}(X)$ on $X=X_n[m_1M_1,\dots , m_kM_k]$ with respect to which we consider aCM sheaves on $X$. Notice, moreover, than on $X$, as on any hypersurface, any line bundle $\Oo_X(t)$ is aCM. In particular, the structure sheaf $\Oo_{M_i}$ of the reduction of any of the components of $X$ is again aCM as an $\Oo_X$-sheaf.

As a special case, let us assume that $k=1$; for a fixed hyperplane $H_n \subset \PP^{n+1}$ with $n\ge 2$, define a projective scheme $X_n[m]:=X_n[mH_n]$ to be the effective Cartier divisor $mH_n$ for $m>0$. It has the dualizing sheaf $\omega _{X_n[m]} \cong \Oo _{X_n[m]}(m-n-2)$. For example, in case $n=2$, we have $X=X_2[2]$ the double plane. If $w$ is a defining equation of $H_n$, then for $m\ge 2$, the restriction map $\Oo_{X_n[m]} \rightarrow \Oo_{X_n[m-1]}$ has kernel isomorphic to $\Oo_{H_n}(1-m)$:
\begin{equation}\label{eqgl}
0\to \Oo_{H_n}(1-m) \to \Oo_{X_n[m]} \to \Oo_{X_n[m-1]} \to 0.
\end{equation}

If $\Ee$ is an aCM sheaf on $X_n[m]$, then it is given as the cokernel of the map defined by the matrix $\mathrm{M}$ in (\ref{mf}) with $\det (\mathrm{M})=w^{ma}$ for some $a\in {\left(\frac{1}{m}\right)}{\NN}$. Note that by Definition \ref{def}, the rank of $\Ee$ belongs to ${\left(\frac{1}{m}\right)}{\NN}$.

\begin{remark}
For an aCM sheaf $\Ee\cong \Oo_H(a)$ of rank $1/2$ on the double plane $X=X_2[2H]$, the corresponding matrix in (\ref{mf}) is $\mathrm{M}=(w)$. Thus $\Ee$ is isomorphic to the cokernel of the map $\Oo_{\PP^3}(-1)\rightarrow \Oo_{\PP^3}$ given by the multiplication by $w$, up to twist. Recall that $H$ is the plane given by $w=0$.
\end{remark}

Now we introduce a special type of coherent sheaves as in \cite[1st Definition at page 318]{ch} and \cite[Definition 6.5]{CH}.

\begin{definition}\label{llay}
Fix $r\in {\left(\frac{1}{m}\right)}{\NN}$. A coherent sheaf $\Ee$ of rank $r$ on $X_n[m]$ is said to be {\emph{layered}} if there exists a filtration $0 =\Ee _0\subset \Ee _1\subset \cdots \subset \Ee _{mr-1} \subset \Ee _{mr} =\Ee$ of $\Ee$ with $\Ee _i/\Ee _{i-1} \cong \Oo _{H_n}(a_i)$ with $a_i\in \ZZ$ for all $i=1,\dots ,mr$.
\end{definition}

\noindent It is automatically true by definition that any layered coherent sheaf on $X_n[m]$ is aCM.

We end the section with a technical Lemma:

\begin{lemma}\label{ttt}
We have
$$\Ext_{\PP^{n+1}}^1(\Oo_{H_n}(a), \Oo_{X_n[m]})\cong H^0(\Oo_{H_n}(1-a))$$
for all $a,m \in \ZZ$ with $m>0$.
\end{lemma}

\begin{proof}
Applying the functor $\Hom_{\PP^{n+1}}(\Oo_{H_n}(a), -)$ to
\begin{equation}\label{3e3}
0\to \Oo_{\PP^{n+1}}(-m)\to \Oo_{\PP^{n+1}}\to \Oo_{X_n[m]}\to 0,
\end{equation}
we get
\begin{equation}
\begin{split}
0&\to \Hom_{\PP^{n+1}}(\Oo_{H_n}(a), \Oo_{X_n[m]}) \to \Ext_{\PP^{n+1}}^1 (\Oo_{H_n}(a), \Oo_{\PP^{n+1}}(-m)) \\
&\to \Ext_{\PP^{n+1}}^1 (\Oo_{H_n}(a), \Oo_{\PP^{n+1}}) \to \Ext_{\PP^{n+1}}^1(\Oo_{H_n}(a), \Oo_{X_n[m]}) \\
&\to \Ext_{\PP^{n+1}}^2(\Oo_{H_n}(a), \Oo_{\PP^{n+1}}(-m)),
\end{split}
\end{equation}
where the last term $\Ext_{\PP^{n+1}}^2(\Oo_{H_n}(a), \Oo_{\PP^{n+1}}(-m)) \cong H^{n-1}(\Oo_{H_n}(a+m-n-2))^\vee$ is trivial. The first map is an isomorphism, because we have the following from (\ref{eqgl}) and Serre's duality
\begin{align*}
&\Hom_{\PP^{n+1}}(\Oo_{H_n}(a), \Oo_{X_n[m]}) \supseteq \Hom_{\PP^{n+1}}(\Oo_{H_n}(a), \Oo_{H_n}(1-m)) \cong H^0(\Oo_{H_n}(1-a-m))\\
&\Ext_{\PP^{n+1}}^1(\Oo_{H_n}(a), \Oo_{\PP^{n+1}}(-m))\cong H^n(\Oo_{H_n}(a+m-n-2))^\vee \cong H^0(\Oo_{H_n}(1-a-m)).
\end{align*}
Now the assertion follows from the isomorphism \[\Ext_{\PP^{n+1}}^1(\Oo_{H_n}(a), \Oo_{\PP^{n+1}}) \cong H^0(\Oo_{H_n}(1-a)).\qedhere\]
\end{proof}

\section{aCM vector bundles on hyperplane arrangements with multiplicities}

In this section we are going to consider aCM vector bundles $\Ee$ of rank $r\geq 2$ on $X_n[m]$ with $m\ge 2$.

\begin{lemma}\label{g1}
For two integers $m,n \ge 2$, and for each integer $s\ge 2n+2$, there exist in $\PP^s$
\begin{itemize}
\item a smooth hypersurface $Y$ of degree $m$, and
\item an $(n+1)$-dimensional linear subspace $M$
\end{itemize}
such that $Y\cap M =X_n[m]$.
\end{lemma}

\begin{proof}
We fix homogeneous coordinates $[x_0:\ldots : x_s]$ of $\PP^s$ and write $M:= \{x_{n+2}= \cdots = x_s=0\}$ and $H_n:= \{x_{0}= x_{n+2}=\cdots =x_s=0\} \subset M$, i.e. $H_n$ is the hyperplane in $M$ defined by $x_0=0$. Consider the smooth hypersurface $Y=V(f)\in |\Oo _{\PP^s}(m)|$ for 
\begin{equation}\label{eqaaabb1}
f =x_0^m+\sum _{k =1}^{n+1}  x^{m-1}_k x_{k+n+1} +\sum _{i\ge n+2} x_i^m.
\end{equation}
Then we have $Y\cap M =X_n[m]$.
\end{proof}

\begin{remark}
The assertion in Lemma \ref{g1} was originally proved by dimension counting, while the current proof using the explicit equation (\ref{eqaaabb1}) for a smooth hypersurface is given by the referee. 
\end{remark}

\begin{lemma}\label{g2}
Fix a hypersurface $Y$ and a hyperplane $M$ in $\PP^s$ with $s\ge 3$ such that $M$ is not a component of $Y$. For an aCM vector bundle $\Ee$ of rank $r$ on $Y$, if $\Ee _{|Y\cap M}\cong \oplus _{i=1}^{r} \Oo _{Y\cap M}(a_i)$ for some $a_i\in \ZZ$, then we have $\Ee \cong \oplus _{i=1}^{r} \Oo _Y(a_i)$
\end{lemma}

\begin{proof}
We may assume $0=a_1\ge \cdots \ge a_r$ and let $e$ be the number of indices $i$ with $a_i=0$.
Then we have $h^0(\Ee _{|Y\cap M})=e$ and $h^0(\Ee _{|Y\cap M}(-1)) =0$. Since $\Ee$ is aCM, we get that the restriction map $\rho : H^0(\Ee )\rightarrow H^0(\Ee _{|Y\cap M})$ is bijective. By Nakayama's lemma we also get that the natural map $\eta :H^0(\Ee )\otimes \Oo _Y\rightarrow \Ee$ is injective and that its image at each $p\in (Y\cap M)_{\mathrm{red}}$ spans an $e$-dimensional linear subspace of the fiber $\Ee _{p}$.

If $e=r$, we get that $\eta$ is an isomorphism, because $Y\cap M$ is an ample divisor of $Y$ and an injective map between two vector bundles with the same rank is either an isomorphism or drops rank on a hypersurface, which must intersect $Y\cap M$.

Now assume $e<r$ and let $e'$ be the number of indices $i$ with $a_i=a_{e+1}$, i.e. the number of second biggest numbers among $a_i$'s. Set $\Ff := \mathrm{Im}(\eta )\cong \Oo_Y^{\oplus e}$. Since $\Ee/\Ff$ is locally free at each point of $(Y\cap M)_{\mathrm{red}}$, it is locally free outside a finite set disjoint from $Y\cap M$. Now the natural map $\rho ': H^0(\Ee (-a_{e+1})) \rightarrow H^0(\Ee_{|Y\cap M} (-a_{e+1}))$ is surjective, because $\Ee$ is aCM. From $H^0(\Ff _{|Y\cap M}(-a_{e+1}))\cong H^0( \Oo _{Y\cap M}(-a_{e+1})^{\oplus e})$, we see that there is an $e'$-dimensional linear subspace $V$ of $H^0(\Ee (-a_{e+1}))$ such that $V\cap H^0(\Ff (-a_{e+1})) =0$
and the map
$$H^0(\Ff_{|Y\cap M} (-a_{e+1}))\oplus V \to H^0(\oplus _{i=1}^{e+e'}\Oo _{Y\cap M}(a_i-a_{e+1}))=H^0(\Ee_{|Y\cap M}(-a_{e+1}))$$
is bijective. Then we get a map $\eta ': \Oo_Y(a_1)^{\oplus e}\oplus \Oo _Y(a_{e+1})^{\oplus e'} \rightarrow \Ee$ that is injective and of rank $e+e'$ at each point of  $(Y\cap M)_{\mathrm{red}}$. If $r=e+e'$, then we can conclude that the map $\rho '$ is an isomorphism. If $r>e+e'$, we continue in the same way using $\Ee (-a_{e+e'+1})$.
\end{proof}

\begin{lemma}\label{g4}
Let $Y$ be a hypersurface of $\PP^s$ and $M\subset \PP^s$ a linear subspace not contained in $Y$. If $\Ee$ is a Ulrich vector bundle on $Y$, then so is $\Ee _{|Y\cap M}$.
\end{lemma}

\begin{proof}
Note by assumption that we have $\deg (Y\cap M)=\deg (Y)$. First assume that $\dim M =s-1$. Let us consider the exact sequence
 $$0\to\Ee(-1)\to \Ee\to\Ee _{|Y\cap M}\to 0.$$ 
 Since $h^i(\Ee(t))=0$ for any integer $t$ and for any $i=1,\dots s-2$ we get that $\Ee _{|Y\cap M}$ is aCM. Moreover, since $\Ee$ is  initialized, we get $h^0(\Ee _{|Y\cap M}(-1))=0$ and $h^0(\Ee _{|Y\cap M}) =h^0(\Ee )$. This implies that $\Ee _{|Y\cap M}$ is Ulrich. If $\dim M \le s-2$, then we take a hyperplane of $\PP^s$ containing $M$ and use induction on the codimension of $Y$.
\end{proof}

Thus, when there exists an integer $s\ge 2n+2$ such that each smooth hypersurface $Y\subset \PP^s$ of degree $m$ supports an aCM vector bundle of rank $r$, which is not a direct sum of line bundles, we may apply Lemmas \ref{g1}, \ref{g2} and \ref{g4}. But one quite often does not have an existence result of aCM vector bundles that do not split with prescribed rank on all smooth hypersurfaces of degree $m$ in a given projective space; for an existence result about the general hypersurface, see \cite{ESW}. It is sufficient to have the existence of an aCM (or Ulrich) vector bundle of rank $r$ on the hypersurface whose equation is given in the proof of Lemma 3.1.

If $m>2$, the bounds $2n+2$ in Lemma\ \ref{g1} are not enough to ensure that a general hypersurface of degree $m$ in $\PP^s$ contains some $X_n[m]$. In the following Lemma we require much higher bounds for dimension of the ambient projective space:

\begin{lemma}\label{g5}
For two integers $n\ge 2$ and $m\ge 3$, set
$$N_{\mathrm{gen}}(n,m):=\min \left \{ k>0~\bigg|~ (k+1)(k-n-1)\ge \binom{m+n+1}{n+1}\right \}.$$
For each integer $s\ge N_{\mathrm{gen}}(n,m)$ and a general hypersurface $Y\subset \PP^s$ of degree $m$, there exists an $(n+1)$-dimensional linear subspace $M\subset \PP^s$ such that $Y\cap M=X_n[m]$.
\end{lemma}

\begin{proof}
We fix an $(n+1)$-dimensional linear subspace $V\subset \PP^s$ and a hyperplane $H^n\subset V$ over which we consider $X_n[m]$. With the linear systems $E$ and $E'$ in the proof of Lemma \ref{g1}, we consider two natural maps induced by the action of $\mathrm{SL}(s+1)$ on $\PP^s$:
$$u: E\times \mathrm{SL}(s+1)\to |\Oo _{\PP^s}(m)| ~~,~~u': E'\times \mathrm{SL}(s+1)\to |\Oo _{\PP^s}(m)|.$$
The lemma is equivalent to saying that $u$ is dominant. Since $E$ is irreducible and $E'\subset E$, it is sufficient to prove that $u'$ is dominant. Now we have $(s+1)(s-n-1)\ge \binom{m+n+1}{n+1}$. This implies that the map $u'$ is dominant by \cite{DL}.
\end{proof}

Now we generalize the previous set-up further to an arbitrary finite union of hyperplanes of $\PP^{n+1}$ with prescribed multiplicities $X=X_n[m_1M_1,\dots ,m_kM_k]$. We may modify the proofs of Lemmas \ref{g1} and \ref{g5} to get the following result.

\begin{proposition}\label{g6}
For two integers $n,m\ge 2$ and any $k$ integers $m_1, \ldots, m_k$ whose sum is $m$, we get the following over every possible $X=X_n[m_1M_1,\dots ,m_kM_k]$ embedded linearly in $\PP^s$.
\begin{itemize}
\item [(i)] If $s\ge 2n+2$ and every smooth hypersurface of degree $m$ in $\PP^s$ has a rank $r$ aCM (resp. Ulrich) vector bundle that does not split, then the same holds for $X$.
\item [(ii)] If $s\ge N_{\mathrm{gen}}(n,m)$ and a general hypersurface of degree $m$ in $\PP^s$ has a rank $r$ aCM (resp. Ulrich) vector bundle that does not split, then the same holds for $X$.
\end{itemize}
\end{proposition}

\begin{remark}\label{g8}
It has been conjectured in \cite{bgs} that, for smooth hypersurfaces, the rank of aCM (or Ulrich) vector bundles should be at least $\lfloor\frac{s-2}{2}\rfloor$. The conjecture is sharp and has been proved on hyperquadrics and for rank $2$ and $3$ vector bundles; see \cite{trip} and references therein. Proposition \ref{g6} asserts that if for a certain $r\ge 2$ there are aCM (or Ulrich) vector bundles of rank $r$ that do not split on a general hypersurface of degree $m$ in $\PP^s$ with $s\gg 0$, then the same is true for $X_n[m_1M_1,\dots ,m_kM_k]$.
\end{remark}

Now we consider the case $n=2$ and $m_i=1$ for at least one index $i$, i.e. we assume
$\dim (X)=2$ and $X_{\mathrm{reg}} \ne \emptyset$. The case $m=2$, i.e. $X$ is the union of two distinct planes in $\PP^3$, of the following example appears in \cite[Example 4.1]{bhp}. In fact, in the following example, we require for our surface $X\subset \PP^3$ only to be any arbitrary surface with at least one irreducible component of $X_{\mathrm{red}}$ appearing with multiplicity one in $X$.

\begin{example}\label{bo1}
Let $X\subset \PP^3$ be a surface of degree $m\ge 2$ with $X_{\mathrm{reg}} \ne \emptyset$. For a fixed point $p \in X_{\mathrm{reg}}$, we have
$$\Ext_X^1 (\Ii_{p,X}(1), \Oo_X(m-2))\cong H^1(\Ii_{p,X}(-1))^\vee \cong \mathbf{k}$$
by Serre's duality. So up to isomorphisms there exists a unique non-trivial extension $\Ee$ as an $\Oo _X$-sheaf, fitting into the exact sequence
\begin{equation}\label{eqb1}
0 \to \Oo _X(m-2) \to \Ee\to \Ii _{p,X}(1)\to 0.
\end{equation}
Such sheaf $\Ee$ is uniquely determined by the point $p$ and it is locally free outside $p$.

\quad \emph{Claim 1:} $\Ee$ is locally free of rank two.

\quad \emph{Proof of Claim 1:} Obviously $\Ee$ has rank two on each of the components of $X$ and it is locally free outside $p$. Note that $X$ is Gorenstein and by
assumption $p$ is a smooth point of $X$. Since $\omega _X\cong \Oo _X(m-4)$, we have $h^0(\omega _X\otimes \Oo _X(3-m)) =0$. So the Cayley-Bacharach condition is satisfied and $\Ee$ is locally free;  see \cite{cat}, where one only considers the case in which $X$ is a Gorenstein normal surface, but in this particular case with $p\in X_{\mathrm{reg}}$ we can adapt the proof in \cite{cat} or the classical proof using the duality of the Cayley-Bacharach property; also we only need one implication of Cayley-Bacharach, not the ``~if and only if~'' statement.

\quad \emph{Claim 2:} $\Ee$ is aCM.

\quad \emph{Proof of Claim 2:} Since $h^1(\Ii _{p,X}(i)) =0$ for all $i\ge 0$, (\ref{eqb1}) gives $h^1(\Ee (i)) =0$ for all $i\ge -1$. Its dual asserts that $h^1(\Ee ^\vee (i)) =0$ for all $i\le 3-m$. From $\det (\Ee )\cong \Oo _X(m-1)$ we have $\Ee^\vee \cong \Ee (1-m)$. Thus we get $h^1(\Ee (i)) =0$ for all $i\le 2$, concluding the proof of Claim 2.

On the other hand, from (\ref{eqb1}) we have $h^0(\Ee (1-m)) =0$ and $h^0(\Ee (2-m))> 0$. In particular, $\Ee$ is not Ulrich for $m\ge 3$, while in case $m=2$ it is Ulrich.

\quad \emph{Claim 3:} There are no integers $a, b$ such that $\Ee \cong \Oo _X(a)\oplus \Oo _X(b)$.

\quad \emph{Proof of Claim 3:} Here we use that $m\ge 2$, because for $m=1$ we would just get the trivial vector bundle of rank two. Assume that such $a,b$ exist, say $a\ge b$. Since $h^0(\Ee (2-m)) =1$
and $h^0(\Ee (1-m)) =0$, we get $(a,b)=(m-2,1)$ and $m\ge 3$. Then we get $h^0(\Ee) = \binom{m+1}{3} +4$, while (\ref{eqb1}) gives $h^0(\Ee )=\binom{m+1}{3} +3$.
\end{example}

\begin{remark}
In Example \ref{bo1} we do not claim that $\Ee$ does not split, e.g. if $X$ is a smooth quadric, then $\Ee$ is the direct sum of the two spinor line bundles of $X$, up to a twist. If $\mathrm{Pic}(X) \cong \ZZ \langle \Oo _X(1)\rangle$, then $\Ee$ is indecomposable; it happens when $X=X_2[m_1M_1,\dots
,m_kM_k]$ by Lemma \ref{bo2}.
\end{remark}

By Lemma \ref{bo2} and Example \ref{bo1} we immediately get the following result.
\begin{corollary}\label{poppp}
Let $X=X_2 [m_1M_1,\dots ,m_kM_k] \subset \PP^3$ be a union of planes with $m_i=1$ for at least one index $i$. Then there exists an indecomposable aCM vector bundle of rank two on $X$.
\end{corollary}

\begin{remark}
Take $X$ as in Corollary \ref{poppp} with $m> 3$. For any $p\in X_{\mathrm{reg}}$ call $\Ee_p$ the aCM vector bundle of rank two in (\ref{eqb1}). Since $m\ge 3$, we have $h^0(\Ee _p(2-m)) =1$, in particular (\ref{eqb1}) is the Harder-Narasimhan filtration of $\Ee _p$. By  uniqueness of the Harder-Narasimhan filtration, we have $\Ee_p \not \cong \Ee_q $ for $p\ne q\in X_{\mathrm{reg}}$. For large $m$ we may find $X$ with finite automorphism groups. Hence we get a $2$-dimensional family of aCM vector bundles of rank two, even after the action of $\mathrm{Aut}(X)$.
\end{remark}


\section{aCM vector bundles on the double plane}
In this section, we discuss the aCM vector bundles on $X:=X_2[2]$, i.e. the double plane in $\PP^3$. Set $S=\mathbf{k}[x,y,z,w]$ and assume that $X$ is given by $f=w^2$, i.e. $X$ is the double plane whose reduction is the plane $H$ given by $w=0$. By Lemma \ref{bo2}, we have $\mathrm{Pic}(X)=\ZZ$ generated by $\Oo_X(1)$ associated to the hyperplane class and $\omega_X\cong \Oo_X(-2)$; see also \cite[Example 5.10]{HK}. Since every line bundle on $H$ is aCM, every line bundle on $X$ is also aCM due to the following exact sequence
\begin{equation}\label{eqa1}
0\to \Oo_H(-1) \to \Oo_X \to \Oo_H \to 0.
\end{equation}
In particular, every direct sum of line bundles on $X$ is aCM. Note also that any extension of an aCM vector bundle $\Ee$ by a line bundle on $X$ splits, because $\Ext^1(\Ee, \Oo_X(k))\cong H^1(\Ee(-k-2))^\vee$ is trivial for any $k\in \ZZ$. By Definition \ref{def} we get $\mathrm{rank}(\Ee)$ is in $r\in {\left(\frac{1}{2}\right)}{\ZZ}$. The ideal sheaf $\Ii_H$ of $H$ in $X$ is $\Oo_H(-1)$ in (\ref{eqa1}), in particular it is an aCM sheaf of rank $1/2$.

If $\Ee$ is a vector bundle on $X$, then by tensoring (\ref{eqa1}) with $\Ee$ we get the following exact sequence
\begin{equation}\label{eqa111}
0\to \Ee(-1)_{|H} \to \Ee \to \Ee_{|H} \to 0.
\end{equation}

In the next lines we gather some crucial properties of minimally regular aCM sheaves on $X$ that will be use for the rest of the paper.
If $\Ee$ is such a sheaf on $X:=X_2[2]$, it follows immediately from the definition that $\Hom_X(\Ee,\Oo_X(1))\cong H^n(\Ee(-3))\neq 0$, so there exists an exact sequence of $\Oo _X$-sheaves
\begin{equation}\label{eqerr3}
0 \to \Ll \to \Ee \stackrel{\pi}{\to} \Ii _A(1) \to 0
\end{equation}
with $A$ a proper closed subscheme of $X$ and $\Ll:= \ker (\pi)$. The sheaf $\Ii _A(1)$ is a nonzero subsheaf of $\Oo _X(1)$, in particular it has positive depth. Let $\langle A \rangle \subseteq \PP^3$ denote the linear span of $A$. Since $\Ee$ is $0$-regular, it is globally generated. This implies that $\Ii _A(1)$ is also globally generated. Thus we get $ A =\langle A\rangle \cap X$ as schemes. Taking the possible linear subspaces $\langle A\rangle$, we get that $A$ is one of the schemes in the following list:

\begin{possibility}\label{poss}
\begin{itemize}
\item[(a)] $A=\emptyset$ and $\Ii_A(1)\cong \Oo_X(1)$.
\item[(b)] $A = H$ and $A$ is  cut out in $X$ by only one linear equation; $\Ii_A(1)\cong \Oo_H$.
\item[(c)] $A$ is a line $L\subset H$; $A$ is  cut out in $X$ by the plane $H$ and another plane that is different from $H$.
\item[(d)] $A$ is a connected scheme of degree two with $A_{\mathrm{red}}=\{p\}$ a point; $A$ is a complete intersection of $X$ with the line $\langle A \rangle \not\subseteq H$.
\item[(e)] $A=\{p\}$ for some point $p\in H$.
\item[(f)] $A = X\cap M$ with $M\subset \PP^3$ a plane that is different from $H$; $\Ii_A(1)\cong \Oo_X$.
\end{itemize}
\end{possibility}
\begin{remark}
  If $A\supseteq H$, then we have $\Ii_A(1) \subset \Oo_H$. Since $\Ii_A(1)$ is globally generated, we have $\Ii_A(1) \cong \Oo_H$. So we get the case (b).
\end{remark}

\begin{lemma}\label{kernelacm}
Let $\Ee$ be a minimally regular aCM sheaf on $X$ and consider the associated short exact sequence (\ref{eqerr3}). Then:
 \begin{itemize}
   \item [(i)] $\Ii_A$ is a $1$-regular $\Oo_X$-sheaf.
   \item [(ii)] If $\dim A\neq 0$, then $\Ll$ is $0$-regular and aCM sheaf.
   \item [(iii)]If $\dim A= 0$, then $\Ll$ is $1$-regular and aCM sheaf.
 \end{itemize}
\end{lemma}
\begin{proof}
\quad {(i)} From the surjective map $H^2(\Ee(t))\rightarrow H^2(\Ii_A(t+1))$ we see that $H^2(\Ii_A(t+1))=0$ for $t\geq -2$. On the other hand, from the short exact sequence
\begin{equation}\label{definA}
    0\to\Ii_A(t)\to\Oo_X(t)\to\Oo_A(t)\to 0
\end{equation}
\noindent we also obtain $H^1(\Ii_A(t))=0$ for $t=0$ ($A$ is always connected), so $\Ii_A$ is $1$-regular. In particular, $H^1(\Ii_A(t))=0$ for $t\geq 0$. Indeed, notice that in the cases when $\dim A\neq 0$, the ideal sheaf $\Ii_A$ is an aCM $\Oo_X$-sheaf.

\quad {(ii)} Again from (\ref{eqerr3}) we get $H^1(\Ll(t))=0$ for $t<0$. Now, from the isomorphism $0=H^1(\Ii_A(-1))\rightarrow H^2(\Ll(-2))\rightarrow 0$, we get at once that $\Ll$ is $0$-regular. Since clearly $\Ll$ also has positive depth at any point $x\in X$, we can conclude that $\Ll$ is an aCM sheaf.

\quad {(iii)} In this case, we only get $0=H^1(\Ii_A)\cong H^2(\Ll(-1))$. But, on the other hand, note that $\Ii_A(1)$ is generated by the image of $H^0(\Ee)$ and $\Ii_A(1)$ needs two sections to be generated. This implies $h^1(\Ll)=0$ from $h^1(\Ee)=0$. Therefore $\Ll$ is $1$-regular. Since we have $H^1(\Ll(t))=0$ for $t<0$ in any case, $\Ll$ is also aCM.
\end{proof}

\begin{remark}\label{trivialities}
From the previous Lemma, we see that in cases (a) and (f), the extension should be trivial and therefore $\Ee\cong\Ll\oplus\Ii_A(1)$.
\end{remark}

\begin{lemma}\label{gg0}
Let $\Ee$ be an aCM vector bundle of rank $r$ on $X$. If $\Ee _{|H}$ splits, then $\Ee$ also splits.
\end{lemma}
\begin{proof}
Let $\Ee$ be an aCM vector bundle on $X$ such that $\Ee_{|H} \cong \oplus_{i=1}^r \Oo_{H}(a_i)$ with $a_1\geq \cdots \geq a_r$.
Up to a twist we may assume that $\Ee$ is minimally regular and therefore from the long exact sequence associated to (\ref{eqa111}) we get $a_r \ge 1$. This implies that $\Ee(-1)_{|H}$ is globally generated.

Note that the restriction map $H^0(\Ee(-1)) \rightarrow H^0(\Ee(-1)_{|H})$ is surjective and $H^0(\Ee(-1)_{|H})$ globally generates $\Ee(-1)_{|H}$. Take a nonzero section $s$ in $H^0(\Ee(-1)_{|H})$ induced from the last factor $\Oo_H(a_r-1)\cong \Oo_{H}$. The section $s$ lifts to a section $\sigma \in H^0(\Ee(-1))$ such that $\sigma_{|H}=s$. Since $s$ vanishes at no point of $H$, $\sigma$ also vanishes at no point of $H$. Let $j: \Oo _{X_2[2]}\to \Ee$ be the map induced by $\sigma$. Set $\Gg := \mathrm{coker}(j)$.

\quad \emph{Claim:} The map $j$ is injective and $\Gg$ is a a vector bundle of rank $r-1$ on $X_2[2]$.

\quad \emph{Proof of Claim:} For a fixed point $p\in H$, let $\Ee (-1) _p$ and $\Gg _p$ denote the stalks of $\Ee$ and $\Gg$ at $p$, respectively. Let $j_p: \Oo _{X_2[2],p} \rightarrow \Ee (-1)_p$ be the map induced by $\sigma$. Since $\Ee$ is locally free, there is an isomorphism $ \phi : \Ee (-1)_p \rightarrow \Oo _{X_2[2],p}^{\oplus r}$ of $\Oo _{X_2[2],p}$-modules. The map $\phi \circ j_p
: \Oo _{X_2[2],p} \to \Oo _{X_2[2],p}^{\oplus r}$ is given by some $a = (a_1,\dots ,a_r)\in \Oo _{X_2[2],p}^{\oplus r}$. The condition $\sigma ({p})\ne 0$ is equivalent
to the existence of $i\in \{1,\dots ,r\}$ such that $a_i$ is not contained in the maximal ideal of $\Oo _{X_2[2],p}$. Thus $\phi \circ j_p$ is injective and $\Gg _p$
is isomorphic  to a direct factor of $\Oo _{X_2[2],p}^{\oplus r}$ (and hence it is locally free) with rank $r-1$. \qed

The restriction to $H$ of the injective map $j$ maps $\Oo _{H}(1)$ onto a factor of $\Ee_{|H}$. Thus $\Gg _{|H}$ is isomorphic to a direct sum of line bundles on $H_n$, hence it is aCM. The exact sequence (\ref{eqa111}) with $\Gg$ instead of $\Ee$ gives that $\Gg$ is aCM. By induction on the rank we get that $\Gg$ is isomorphic to a direct sum of line bundles. Thus $\Ee$ is isomorphic to a direct sum of line bundles.
\end{proof}

\begin{remark}\label{t1.1}
An $\Oo _{\PP^3}$-sheaf $\Ee$ is an $\Oo _X$-sheaf if and only if $w^2\Ee=0$, i.e. $w\mathrm{Im}(f_w)=0$, where the map $f_w: \Ee \rightarrow \Ee(1)$ is induced by the multiplication by $w$;
\begin{equation}\label{999}
0\to \ker f_w\to \Ee\to\mathrm{Im} f_w \to 0.
\end{equation}
In particular, any $\Oo_X$-sheaf admits an extension of an $\Oo_H$-sheaf by another $\Oo_H$-sheaf. If $\Ee$ is a locally free $\Oo_X$-sheaf, then this exact sequence is exactly (\ref{eqa111}). Now consider any $\Oo_{\PP^3}$-sheaf $\Ee$ that is an extension, as an $\Oo_{\PP^3}$-sheaf, of an $\Oo_H$-sheaf $\Ee_1$ by another $\Oo_H$-sheaf $\Ee_2$. We claim that $\Ee$ is an $\Oo_X$-sheaf. Indeed, $\mathrm{Im} (f_w)$ is isomorphic to a subsheaf of $\Ee_2(1)$. So we get $w\mathrm{Im}(f_w)=0$. Thus $\Ee$ is an $\Oo _X$-sheaf..
\end{remark}



\begin{remark}\label{s+3}
Let $\Ee$ be an aCM sheaf of rank $r\in {\left(\frac{1}{2}\right)}{\ZZ}$ and let $\Ff_1:= \oplus _{i=1}^{s} \Oo _H(t_i)$ with $t_1\ge \cdots \ge t_s$ be the kernel of the map $f_{w,\Ee} : \Ee \rightarrow \Ee (1)$ induced by the multiplication by $w$ as in Remark \ref{t1.1}. The image $\Ff_2:=\mathrm{Im}(f_{w, \Ee})\subset \Ee (1)$ is a torsion-free sheaf on $H$ with rank $2r-s$. Since $w^2=0$, we have $\Ff_2 \subseteq  \Ff_1 (1)$, in particular we get $s\ge r$. We obviously have $s\le 2r $, and $s=2r$ if and only if $\Ee \cong \oplus _{i=1}^{s} \Oo _H(t_i)$.
\end{remark}

Our next goal is to prove Theorem \ref{pprop}, namely that an aCM vector bundle $\Ee$ of rank two on $X$ splits.

\begin{remark}\label{erra16}
We know that $\Ee$ fits in the exact sequence (\ref{eqerr3}) with $A$ a subscheme from the list of possible cases \ref{poss}. In case (a) we have $\Ii_A(1) \cong \Oo_X(1)$. In case (f), $A$ is an effective divisor in $|\Oo_X(1)|$, in particular we have $\Ii _A(1)\cong \Oo _X$. In either case, the vector bundle $\Ee$ corresponds to an element in $\Ext_X^1 (\Oo_X(i), \Oo_X(c_1-i))=0$ with $i\in \{0,1\}$. Thus $\Ee$ splits. In case (b), we have $\Ii _A(1) =\Oo_H$. Since the tensor product is a right exact functor, the surjection $\pi: \Ee \rightarrow \Oo _H$ induces a surjection $\pi_1: \Ee _{|H}\rightarrow \Oo_H$. Since $\Ee$ is globally generated, $\Ee _{|H}$ is also globally generated. The surjection $\pi_1$ implies that $\Ee _{|H}$ splits. Then by Lemma \ref{gg0}, $\Ee$ splits. In case (c), since the tensor product is a right exact functor, the surjection $\pi$ induces a surjection $\Ee _{|H} \rightarrow \Ii _L(1)\otimes \Oo _H$. The $\Oo _H$-sheaf $\Ii _L(1)\otimes \Oo _H$ has a torsion part $\tau$ supported on $L$ and $(\Ii _L(1)\otimes \Oo _H) /\tau \cong \Oo _H$. Thus we obtain a surjection $\pi_1: \Ee _{|H} \rightarrow \Oo _H$ and again we obtain  the decomposability of $\Ee$.
\end{remark}

By Remark \ref{erra16} it remains to deal with the cases (d) and (e), concerning the decomposability of $\Ee$. In both cases, then the map $\pi$ induces a surjection $\pi_1: \Ee _{|H} \rightarrow \Ii _{p,H}(1)$. Since $\Ii _{p,H}(1)$ has no torsion and $\Ee_{|H}$ is locally free, we get that $\mathrm{ker}(\pi_1)$ has rank one with pure depth two. Thus $\mathrm{ker}(\pi_1)$ is isomorphic to $\Oo _H(c_1-1)$, where $\det (\Ee)\cong \Oo_X(c_1)$.

\begin{lemma}\label{erra11}
For each integer $a\ge 0$ and a point $p\in H$, there exists a unique vector bundle $\Ee _{p,a}$ of rank two on $H$ fitting into the exact sequence
\begin{equation}\label{eqerr4}
0 \to \Oo _H(a) \to \Ee _{p,a} \to \Ii _{p,H}(1)\to 0,
\end{equation}
up to isomorphism. Here, the vector bundle $\Ee _{p,a}$ is globally generated.
\end{lemma}

\begin{proof}
There exists a vector bundle $\Ee_{p,a}$ fitting into (\ref{eqerr4}), because the Cayley-Bacharach condition is satisfied. And any such sheaf is globally generated, because $\Oo _H(a)$ and $\Ii _{p,H}(1)$ are globally generated and $h^1(\Oo _H(a))=0$. Thus it remains to prove that the vector bundle $\Ee _{p,a}$ is unique, up to isomorphism, and it is sufficient to prove that the dimension of $\Ext_H ^1(\Ii _{p,H}(1),\Oo_H(a)) $ is at most one. This is true, because $h^1(\Oo _H(a-1))=0$, the local $\Ext ^1$-group is the skyscraper sheaf $\mathbf{k}_p$ supported by $p$ and we may use the local-to-global spectral sequence of the $\Ext$-functor.
\end{proof}

\begin{remark}\label{erra12}
In case $a=0$, we have $\Ee_{p,0} \cong \Omega _H^1(2)$ for any choice of $p\in H$.
\end{remark}

\begin{lemma}\label{erra14}
For a fixed integer $a>0$, we have the following:
\begin{itemize}
\item [(i)] we have $h^2(\Ee _{p,a}(t)) =0$ if and only if $t\ge -2$,
\item [(ii)] we have $h^1(\Ee_{p,a}(-1)) >0$, and
\item [(iii)] for any $p\in H$, there is no aCM vector bundle $\Ee $ of rank two on $X$ with $\Ee _{|H} \cong \Ee _{p,a}$.
\end{itemize}
\end{lemma}

\begin{proof}
Since $p$ is zero-dimensional, we have $h^2(H,\Ii _{p,H}(t)) =h^2(H,\Oo _H(t))$ for all $t\in \ZZ$. Then we get (i) and (ii), by using (\ref{eqerr4}). For (iii) assume that such $\Ee$ exists. Then we have $h^2(H,\Ee _{|H}(-2)) =0$ and $h^1(H,\Ee _{|H}(-1)) >0$ by (i) and (ii). Now we may use (\ref{eqa111}) to get a contradiction.
\end{proof}
\begin{proof}[Proof of Theorem \ref{pprop}:]
Let $\Ee$ be a minimally regular aCM vector bundle of rank two on $X$, with $\det (\Ee) \cong \Oo_X(c_1)$. Since $\Ee$ is globally generated, we get $c_1\ge 0$. If $c_1=0$, then we get $h^0(\Ee)=h^0(\Ee^\vee)$ which implies $\Ee \cong \Oo_X^{\oplus 2}$. So we may assume that $c_1\ge 1$. Now by Remark \ref{erra16} it is enough to prove the assertion for the cases (d) or (e) of Possible Cases \ref{poss} for $A$. Let $\tau$ be the torsion part of the sheaf $\Ii _A(1)\otimes \Oo _H$. In both cases we have $(\Ii _A(1) \otimes \Oo _H) /\tau \cong \Ii _{p,H}(1)$. Thus the surjection $\Ee \rightarrow \Ii _A(1)$ induces a surjection $\Ee _{|H} \rightarrow \Ii _{p,H}(1)$, in particular we get $\Ee _{|H} \cong \Ee _{p,c_1-1}$ by Lemma \ref{erra11}. Lemma \ref{erra14} excludes the case $c_1>1$. So the only possibility is $c_1=1$, when we have $\Ee _{|H}\cong \Omega _H^1(2)$. Then we have the exact sequence
\begin{equation}\label{eqerrc1}
0 \to \Omega ^1_H(-1) \to \Ee(-2) \to \Omega ^1_H\to 0.
\end{equation}
Since $h^1(\Omega ^1_H)=1$ and $h^2(\Omega ^1_H(-1)) =h^0(T_H(-2)) =0$, we have $h^1(\Ee(-2) )\ge1$. So $\Ee$ is not aCM.
\end{proof}

On the other hand the assertion in Theorem \ref{pprop} may not hold for higher rank. Indeed, the vector bundle $\Ss_X$ in Example \ref{qwe} gives a counterexample in rank four; see Proposition \ref{errb1}.

\begin{example}\label{qwe}
Let $Q_5\subset \PP^6$ be a smooth quadric hypersurface and $\Ss$ the spinor bundle on $Q_5$ (of rank four); see \cite{Ot2}. Fix a plane $H\subset Q_5$ and take a $3$-dimensional linear space $H\subset V\subset \PP^6$ such that the quadric $Q_5\cap V$ has rank one. So we write $X = Q_5\cap V$ and set $\Ss_X := \Ss _{|X}$. Since $\Ss(1)$ is Ulrich, $\Ss_X(1)$ is also Ulrich by Lemma \ref{g4}. Set $\Ss_H:= \Ss _{|H}$. By \cite[Theorem 2.5]{Ot2} we have $\Ss_H \cong \Oo _H\oplus \Oo _H(-1)\oplus \Omega _H^1(1)$.  Since $\Ss_X$ is locally free, it fits into an exact sequence on $X$
\begin{equation}\label{eqerrb1}
0 \to \Ss_H(-1) \to \Ss_X \to \Ss_H\to 0.
\end{equation}
Let $\Delta$ be the set of isomorphism classes of vector bundles $\Ee$ with $\Ee(1)$ Ulrich and $\Ee _{|H} \cong \Ss _H$. Then each element of $\Delta$ is an extension of $\Ss _H$ by $\Ss _H(-1)$, and we have a map $\Delta \rightarrow \PP \Ext_X^1(\Ss_H, \Ss_H(-1))$. Since $\Ss _X(1)$ is Ulrich, we have $\Delta \ne \emptyset$ and the image of this map contains a non-empty Zariski open subset of $\PP \Ext_X^1(\Ss_H, \Ss_H(-1))$.
\end{example}

\begin{proposition}\label{errb1}
For any $[\Ee]\in \Delta$, $\Ee$ is indecomposable.
\end{proposition}
\begin{proof}
Otherwise, since there is no Ulrich line bundle on $X$, each summand of $\Ee(1)$ would be an Ulrich bundle of rank two. But by Theorem \ref{pprop} any such bundle would split, a contradiction.
\end{proof}

In the previous lines we showed the existence of rank four Ulrich vector bundles on the double plane $X$. On the other hand, we already proved that there are no Ulrich bundles of rank one and two on $X$. Therefore with the next Proposition we show that four is the lowest possible rank for an Ulrich bundle on $X$.

\begin{proposition}\label{noulrich3}
If $\Ee$ is an Ulrich vector bundle of rank $r$ on the double plane $X\subset \PP^3$, then $r$ is divisible by four. 
\end{proposition}
\begin{proof}
Let us suppose that $\Ee$ is an Ulrich vector bundle of rank $r$ on $X$. We know that $\Ee$ is minimally regular and that $h^0(\Ee(-1))=0$ by \cite[Proposition 2.1]{ESW}. Therefore, from the long exact sequence of cohomology groups associated to (\ref{eqa111}) it is immediate to see that $h^1(\Ee_{|H}(t))=0$ for $t\le -3$. From the $0$-regularity of $\Ee$ we have $h^2(\Ee(-2))=0$. Thus we get $h^2(\Ee_{|H}(-2))=0$ and this implies $h^1(\Ee_{|H}(-1))=0$ by (\ref{eqa111}). In particular, $\Ee_{|H}$ is $0$-regular and we get $h^1(\Ee_{|H}(t))=0$ for $t\ge -1$. On the other hand, we have $h^0(\Ee_{|H}(-1))=h^1(\Ee_{|H}(-2))\neq 0$; otherwise, $\Ee$ would split by Lemma \ref{gg0}. Thus by \cite[Theorem 2.2 and Corollary 2.4]{anc} we get $\Ee_{|H} \cong \Omega_H^1(2)^{\oplus a} \oplus \Oo_H(1)^{\oplus a}\oplus \Oo_H^{\oplus b}$ for $a=h^0(\Ee_{|H}(-1))$ and some $b\in \ZZ_{\ge 0}$. In particular, we have
\begin{align*}
2r=6a+2b&=h^0(\Ee)\\
&=h^0(\Ee_{|H}(-1))+h^0(\Ee_{|H})\\
&=a+(6a+b)=7a+b.
\end{align*}
This implies that $a=b$, and in particular we get $r=4a$.  
\end{proof}

While the assertion of Theorem \ref{pprop} does not extend to higher rank, it holds for  higher dimensional quadrics, even with smaller corank.

\begin{lemma}\label{cor3.9bis}
Let $Q\subset \PP^{n+1}$ with $n\ge 3$, be a quadric hypersurface of corank at least $3$. Any aCM vector bundle of rank two on $Q$ splits.
\end{lemma}

\begin{proof}
Let $\Ee$ be an aCM vector bundle of rank two on $Q$.

\quad (a) First assume that $Q$ is a hyperplane with multiplicity two in $\PP^{n+1}$, i.e. $Q =X_n[2]$ for some $n\ge 3$. We choose a three-dimensional linear subspace $V\subset \PP^{n+1}$ so that $V\cap Q$ is a double plane in $V$. Then $\Ee_{|V\cap Q}$ is an aCM vector bundle of rank two on the double plane. This implies that it splits. In particular, its restriction to the reduction of $V\cap Q$, say $H_2:=(V \cap Q)_{\mathrm{red}}$, splits. Moreover $V$ can be chosen so that $H_2$ can be any plane contained in $H_n$, which implies that the splitting type of $\Ee_{|H_2}$ does not change as $H_2$ varies in $H_n$. In particular, $\Ee_{|H_n}$ is a uniform vector bundle of rank two on $H_n$ and this implies that $\Ee_{|H_n}$ splits. Then $\Ee$ also splits due to Lemma \ref{gg0}.

\quad (b) Assume $n=3$. The case in which $Q$ has corank $4$, is true by step (a). In the case when $Q$ has corank $3$, i.e. $Q=M_1\cup M_2$ with $M_1,M_2$ two distinct hyperplanes of $\PP^4$, we may apply \cite[Theorem 3.13]{bhp}.

\quad ({c}) Now assume $n>3$ and that the assertion holds for a lower dimensional projective space. Due to step (a) we may assume that $Q$ is a reduced quadric hypersurface. Take a hyperplane $M\subset \PP^{n+1}$ such that $Q\cap M$ has corank $k+1$, where $k$ is the corank of $Q$. Note that $\Ee
_{|Q\cap M}$ is also aCM. By the inductive assumption we have $\Ee _{|Q\cap M} \cong \Oo _{Q\cap M}(a)\oplus \Oo _{Q\cap M}(b)$ for some integers $a\le b$. Up to a twist we may assume that $b=0$. Since $n>3$ and $Q$ has corank at least $3$, for each $p\in Q$ there is a three-dimensional linear subspace $W\subset Q$ such that $p\in W$. Since $\Ee _{|W\cap M} \cong \Oo _{W\cap M}(a)\oplus \Oo _{W\cap M}$, the argument in step (a) gives $\Ee _{|W}\cong \Oo _W(a)\oplus \Oo _W$. Note that the every point $p\in Q$ is contained in a three-dimensional linear subspace $W\subset Q$. Since $h^1(\Ee (t-1))=0$ for all $t\in\ZZ$, the restriction map $\rho _t: H^0(\Ee (t))\rightarrow  H^0(\Ee _{|Q\cap M}(t)) $ is surjective. Since $h^0(\Ee (t)) =0$ for $t\ll 0$ and $h^0(\Ee_{|Q\cap M}(t-1)) =0$ for all $t\leq 0$, we get that $h^0(\Ee (-1)) =0$ and that $\rho _0$ is bijective. Let $\eta : H^0(\Ee )\otimes \Oo _Q\rightarrow \Ee$ denote the evaluation map.

First assume $a=0$. We get $h^0(\Ee )=2$ and that the evaluation map $\eta$ is an isomorphism
at all points of the ample divisor $Q\cap M$. Since $H^0(\Ee )\otimes \Oo _Q$ and $\Ee$ are vector bundles with the same rank, $\eta$ is an isomorphism.

Now assume $a<0$. We have $h^0(\Ee )=1$. Fix $p\in Q$ and take a three-dimensional linear space $W\subset Q$ such that $p\in W$. Since $\Ee _{|W}\cong \Oo _W(a)\oplus \Oo _W$, $\eta$ induces a map of rank one from the fiber of $H^0(\Ee )\otimes \Oo _Q$ to the fiber $\Ee_p$. Thus $\eta$ is injective and $\Ee /\mathrm{Im}(\eta )$ is a line bundle whose restriction to each $W$ is isomorphic to $\Oo _W(a)$. Thus we get $\Ee /\mathrm{Im}(\eta )\cong \Oo _Q(a)$, in particular $\Ee$ splits.
\end{proof}

\begin{corollary}\label{==1}
Let $Q\subset \PP^{n+1}$ with $n\ge 7$, be any quadric hypersurface. If $\Ee$ is an aCM vector bundle of rank two on $Q$, then it splits.
\end{corollary}

\begin{proof}
If $Q$ has corank at least $3$, then we may use Lemma \ref{cor3.9bis}. Thus we assume that $Q$ has corank at most $2$. In particular, there exists a linear subspace $V\subset \PP^{n+1}$ such that $\dim V =6$ and $V\cap Q$ is a smooth quadric hypersurface of $V$. By \cite{ottav} the restriction $\Ee _{|Q\cap V}$ splits. Now we may proceed as in step ({c}) of the proof of Lemma \ref{cor3.9bis}.
\end{proof}



\section{Wildness of the double plane}

\begin{lemma}\label{pop}
Any sheaf of rank $1/2$ on $X$ with pure depth $2$, is isomorphic to $\Oo_H(a)$ for some $a\in \ZZ$.
\end{lemma}
\begin{proof}
Let $\Ee$ be a sheaf of rank $1/2$ on $X$ with pure depth $2$, in particular it is reflexive by \cite[Theorem 1.9]{HK}. Then $\Ff:=\mathrm{ker}(f_w)$ and $\Gg:=\mathrm{Im}(f_w)\subset \Ee(1)$ as in (\ref{999}) are torsion-free (or trivial) by \cite[Proposition 1.2.9]{BH}. Since $w^2\Ee=0$, $f_w$ is not injective. This implies that $\Ff$ is non-zero. Thus by additivity of the rank, we have $\mathrm{rank}(\Ff)=1/2$ and $\Gg \cong 0$. Now $\Ff$ is a reflexive $\Oo_H$-sheaf of rank one. So we get $\Ff \cong \Oo_H(a)$ for some $a\in \ZZ$.
\end{proof}

\begin{example}\label{ex1}
Fix a plane curve $C\subset H$ and consider its ideal sheaf in $X$ with the exact sequence
\begin{equation}\label{eq4}
0\to \Ii_C\to \Oo_{X} \to \Oo_C \to 0.
\end{equation}
Since $\Ii_{C, H}\cong \Oo_H(-k)$ for $k=\deg (C)$, it is aCM, in particular the map $H^0(\Oo_H(t)) \rightarrow H^0(\Oo_C(t))$ is surjective. Thus we get that the map $H^0(\Oo_{X}(t)) \rightarrow H^0(\Oo_C(t))$ is surjective, because it factors through $H^0(\Oo_H(t))$. This implies that $\Ii_C$ is a non-locally free aCM sheaf of rank one, because $\Oo_{X}$ is aCM. Note that $C$ is not a Cartier divisor of $X$, in particular $\Ii_C$ is not locally free along $C$.

For a fixed plane curve $C\subset H$ of degree $d$, the injection $\Oo_H(-1)\rightarrow \Oo_{X}$ in (\ref{eq4}) factors through $\Ii_C$. Now the cokernel of the map $\Oo_H(-1) \rightarrow \Ii_C$ is $\Ii_{C, H}$, which is $\Oo_H(-d)$. So we get an exact sequence
\begin{equation}\label{eqbbb1}
0\to \Oo_H(-1) \to \Ii_C \to \Oo_H(-d) \to 0.
\end{equation}
By case $m=1$ of Lemma \ref{ttt} for $X$ we get that $\PP \Ext_{\PP^3}^1(\Oo_H(1-d), \Oo_H)\cong \PP H^0(\Oo_H(d))$ and this space parametrizes the plane curves of degree $d$. Therefore $\Ii_C(1)$ determines an element in $\PP \Ext_{\PP^3}^1(\Oo_H(1-d), \Oo_H)$. Notice that indeed we get $\Ext_{\PP^3}^1(\Oo_H(1-d), \Oo_H) \cong \Ext_X^1(\Oo_H(1-d), \Oo_H)$, as it is easily checked applying the functor $\Hom_X(-,\Oo_H(d-1))$ to the surjection $\Oo_X\to\Oo_H$.
\end{example}

\begin{proof}[Proof of Proposition \ref{verywild}:]
Fix a positive integer $k$ and take an integer $d>0$ such that $\binom{d+2}{2} >k$. Let $\Delta \subset \PP H^0(\Oo _H(d))$ be the set parametrizing all smooth curves $C\subset H$ of degree $d$. $\Delta$ is a non-empty Zariski open subset of the projective space $\PP H^0(\Oo _H(d))$ of dimension $\binom{d+2}{2} -1 \ge k$. Thus $\Delta$ is a  non-empty algebraic variety of dimension at least $k$. For any $C\in \Delta$, $C$ is the set of all $p\in H$ at which $\Ii _C$ is not locally free. In particular, if $C, D\in \Delta$ and $C\ne D$, we have $\Ii _C\ncong \Ii _D$. Then we may use the family $\{\Ii _C\}_{C\in \Delta}$ to get the assertion.
\end{proof}

Now we classify aCM sheaf of rank one on $X$ to obtain Theorem \ref{aprop2}.

\begin{proof}[Proof of Theorem \ref{aprop2}:]
Let $\Ee$ be a minimally regular aCM sheaf of rank one on $X$ We get a surjective map $\pi : \Ee \rightarrow \Ii_A(1)$ for a closed subscheme $A\subsetneq X$ in Possible Cases \ref{poss}.

In cases (a), (c), (d), (e) and (f), the surjective map $\pi: \Ee \rightarrow \Ii _A(1)$ is an isomorphism, because $\Ee$ has rank $1$ with pure depth two, in particular it is reflexive by \cite[Theorem 1.9]{HK}. In case (a) and (f), $\Ee$ is isomorphic to $\Oo_X$ and $\Oo_X(1)$, respectively. In case (c), $\Ee\cong \Ii_L(1)$ and $\Ee (-1)$ is as in (iii) with a line as the plane curve; see Example \ref{ex1}. Cases (d) and (e) are excluded, because we have $h^1(\Ii _A(-1)) =\deg (A)$ from the standard sequence for $A\subset X$;
$$0\to \Ii_A(-1) \to \Oo_X(-1) \to \Oo_A(-1) \to 0,$$
and this implies that $A$ is not aCM. Finally in case (b), we have $\Ii_A(1) \cong \Oo_H$ and by Lemma \ref{kernelacm} so by Example \ref{ex1} we get (ii) or (iii).
\end{proof}


Now we discuss wildness in higher rank. For a fixed $r\in (1/2)\NN$ that is at least one, take two positive integers $r_1$ and $r_2$ such that $r_1+r_2=2r$ together with two sequences of integers $\overrightarrow{k}=(k_1, \ldots, k_{r_1})\in \ZZ^{\oplus r_1}$ and $\overrightarrow{m}=(m_1, \ldots, m_{r_2})\in \ZZ^{\oplus r_2}$. Define two vector bundles on $H$ that split as follows:
$$\Aa:=\oplus _{j=1}^{r_1} \Oo _H(k_j)~,~\Bb:=\oplus _{h=1}^{r_2} \Oo _H(m_h).$$
Then $\Gamma:=\PP \Ext_{\PP^3}^1(\Bb, \Aa)$ is of dimension $-1+\sum_{j,h} \max \{ 0, {3+k_j-m_h \choose 2}\}$ by case $m=1$ of Lemma \ref{ttt} for $X$ and each element $\lambda \in \Gamma$ corresponds to a unique aCM sheaf $\Ee _\lambda$ on $X$ of rank $r$, given as an extension of $\Bb$ by $\Aa$. Note that all sheaves $\Ee_{\lambda}$ are layered.

\begin{proposition}\label{iii0}
Fix $r\in (1/2)\NN$, $r \geq 1$, and assume $(r_1, r_2)=(2r-1,1)$ with $\overrightarrow{k}=(k, \ldots, k)$ and $\overrightarrow{m}=(m)$ such that $m < k$ and ${3+k-m\choose 2} \ge 2r-1$. Then for a general $\lambda\in \Gamma$, the sheaf $\Ee _\lambda$ is indecomposable.
\end{proposition}

\begin{proof}
For a general $\lambda \in \Gamma$, set $\Ee := \Ee _\lambda$. Up to a twist, i.e. taking $\Ee$ instead of $\Ee$ for some $t\in \ZZ$, we may assume that $m=0$. We have $\lambda = (\epsilon _1,\dots ,\epsilon _{2r-1})$ with $\epsilon _i\in \Ext_{\PP^3}^1(\Oo _H,\Oo _H(k))$. Since ${3+k\choose 2} \ge 2r-1$ and $\lambda$ is general in $\Gamma$, the extensions $\epsilon _1, \dots ,\epsilon _{2r-1}$ are linearly independent.

Assume that $\Ee \cong \Ff _1\oplus \Ff _2$. Here we consider sheaves such as $\Oo _H(t)$, $\Ee$, $\Ff _1$ and $\Ff _2$ as $\Oo _{\PP^3}$-sheaves, seeing $\Oo _X$ as a quotient of $\Oo _{\PP^3}$. From this point of view these sheaves are pure sheaves of depth $2$ on $\PP^3$ and we may apply the notion of (semi-)stability for pure sheaves; see \cite{Simp}. Note that $0\subset \Aa \subset \Ee$ is the Harder-Narasimhan filtration of $\Ee$, because $k>m=0$ and both of $\Aa$ and $\Bb=\Oo_H$ are semistable. By uniqueness of the Harder-Narasimhan filtrations of $\Ee$ and $\Ff_i$ for each $i$, $\Aa$ must be the direct sum of the first subsheaves of $\Ff_1$ and $\Ff_2$ in their filtrations. Then, due to rank counting, one of the two factors of $\Ee$, say $\Ff_1$, is a factor of $\Aa$. So we have $\Aa \cong \Ff_1 \oplus \Gg$ for some $\Gg$, while the other one $\Ff _2$, is isomorphic to either $\Oo _H$ (the case of $\Aa =\Ff_1$) or an extension of $\Oo _H$ by $\Gg$.

First assume $\Ff _2\not \cong \Oo _H$, that is, $\Gg \ne 0$. Each $\Oo _H(t)$ is simple and there is an integer $s\in \{1,\ldots, 2r-2\}$ such that $\Ff _1 \cong \Oo _H(k)^{\oplus s}$ and $\Gg \cong \Oo _H(k) ^{\oplus (2r-1-s)}$. Taking instead of $\Ff _1$ a direct factor of $\Ff _1$ with minimal rank, it is sufficient to consider only the case $s=1$ for contradiction. Since $\Oo _H(k)$ is simple, we have $\mathrm{Aut}(\Aa) \cong \mathrm{GL}(2r-1,\mathbf{k})$. Hence, up to an element of $ \mathrm{GL}(2r-1,\mathbf{k})$ we may assume that $\Ff _1$ is the first factor of $\Aa$. With this new basis of $\Oo_H(k)^{\oplus (2r-1)}$, set $\lambda=(\overline{\epsilon}_1, \ldots, \overline{\epsilon}_{2r-1})$. Then $\overline{\epsilon}_1$ corresponds to the extension of $\Oo_H$ by $\Oo_H(k)$ with $\Ee /j(\Gg ) \cong \Oo _H(k)\oplus \Oo _H$ as its middle term, where $j: \Gg \rightarrow \Ff_2 \rightarrow \Ee$ is the composition. Thus we get that $\overline{\epsilon}_1$ is zero, contradicting to the linear independence of $\overline{\epsilon}_1, \ldots, \overline{\epsilon}_{2r-1}$.

Now assume $\Gg =0$, that is, $\Ee \cong \Oo _H\oplus \Oo _H(k)^{\oplus (2r-1)}$. The extension class $\lambda$ induces a surjection $\Ee \rightarrow\Oo _H$. Since $k$ is positive, the extension class
$\lambda$ is induced by the projection of $\Oo _H\oplus \Oo _H(k)^{\oplus (2r-1)}$ onto its first factor. Thus we get $\lambda=0$, contradicting Lemma \ref{ttt}.
\end{proof}

\begin{lemma}\label{iii1}
Assume the same numeric invariants as in Proposition \ref{iii0}. For general $\lambda, \lambda'\in \Gamma$, we have $\Ee_{\lambda} \cong \Ee_{\lambda'}$ if and only if there is $g\in \mathrm{GL}(2r-1,\mathbf{k})$ such that $g \cdot \lambda =\lambda'$.
\end{lemma}

\begin{proof}
Set $\Ee:=\Ee_{\lambda}$ and $\Ee':=\Ee_{\lambda '}$. Up to shift, $\Ee$ and $\Ee'$ are indecomposable extensions of $\Bb\cong \Oo _H$ by $\Aa \cong \Oo _H(k)^{\oplus (2r-1)}$, where we have $\mathrm{Aut}(\Aa) \cong \mathrm{GL}(2r-1,\mathbf{k})$ because $\Oo _H$ is simple. Consider all these sheaves as pure sheaves of depth two on $\PP^3$ and use semistability of pure sheaves with respect to the polarization $\Oo _{\PP^3}(1)$. Then $\Aa$ is semistable and the first step of the Harder-Narasimhan filtration of both $\Ee $ and $\Ee '$. Hence every isomorphism $\Ee \rightarrow \Ee '$ induces an automorphism of $\Aa$. The other implication is obvious.
\end{proof}

\begin{corollary}\label{iii2}
For fixed positive integer $n$ and $r\in\left(\frac{1}{2}\right)\NN$, $r\geq 1$, there exists a family $\Delta$ of indecomposable layered aCM sheaves on $X$ of rank $r$ with $\dim \Delta  \ge n$, where each isomorphism class of sheaves appears only finitely many times in $\Delta$.
\end{corollary}
\begin{proof}
By Proposition \ref{iii0} and Lemma \ref{iii1}, a general choice of $(2r-1)$-dimensional subspace in $\Ext_{\PP^3}^1(\Bb, \Aa)$ gives an indecomposable aCM sheaf of rank $r$ on $X$. But the dimension of such choices can be made arbitrarily large by taking $k$ sufficiently large compared to $m$, due to Lemma \ref{ttt}. To get isomorphism classes of sheaves we need to factor by the action of $\mathrm{GL}(2r-1,\mathbf{k})$. We take a general orbit $F$ of this action and choose a variety $\Delta'$ intersecting $F$ transversally and with complementary dimension, so that it intersects $F$ at finitely many points (at least one). Then $\Delta '$ intersects transversally and at finitely many point all fibers near $F$. We take as $\Delta$ a non-empty Zariski open subset of $\Delta'$ intersecting no orbit in a positive dimensional variety.
\end{proof}
\noindent Over $\mathbf{k}=\CC$ we may take instead of $\Delta$ a small Euclidean ball of $\Delta$ and get a one-to-one complex analytic parametrization by a ball in an affine space of dimension equal to $\dim \Delta$.

\section{aCM sheaves of rank $3/2$}
Let us consider the case of an aCM sheaf $\Ee$ on the double plane $X:=X_2[2]$ of rank of $3/2$. We know that $\Ee$ fits in the short exact sequence (\ref{eqerr3}) with $\Ll$ an aCM sheaf and $A$ being one the possible subschemes from the list \ref{poss}. Since we are interested only in indecomposable sheaves, by Remark \ref{trivialities}, we can exclude cases (a) and (f). The following Lemma also allows us to exclude case (d):

\begin{lemma}
There is no aCM sheaf of rank $3/2$ on $X$ fitting on the sequence  (\ref{eqerr3}):
$$
0 \to \Ll \to \Ee \stackrel{\pi}{\to} \Ii _A(1) \to 0
$$
\noindent for $A$ a connected scheme of degree two with $A_{\mathrm{red}}=\{p\}$ a point.
\end{lemma}
\begin{proof}
$\Ll\cong\ker(\pi)$ is aCM by Lemma \ref{kernelacm}. So it is isomorphic to $\Oo_H(a)$ for some $a\in \ZZ$. From the exact sequence
$$0=H^1(\Ff(-1)) \to H^1(\Ii_A)\cong \mathbf{k} \to H^2(\Oo_H(a-1))\to H^2(\Ff(-1))=0,$$
we get $a=-2$. Since $\Ff$ is aCM with $h^2(\Ff(-2))=0$, we get $h^1(\Ii_A(-1))=h^2(\Oo_H(-4))=3$. But since $A$ is a zero-dimensional subscheme of length two, we have $h^1(\Ii_A(-1))=2$, a contradiction.
\end{proof}

It will also be easy to deal with case (c):

\begin{lemma}
  Any non-trivial rank $3/2$ sheaf $\Ee$ on $X$ fitting into the sequence (\ref{eqerr3}) for $A$ as in (c) of Possible Cases \ref{poss}, is a layered sheaf. In particular, it also has a presentation as in (b) of Possible Cases \ref{poss}.
\end{lemma}
\begin{proof}
If $\Ee$ is such a sheaf, by Lemma \ref{kernelacm}, $\Ll$ is an aCM $0$-regular sheaf of rank $1/2$ and therefore $\Ll\cong\Oo_H(a)$ with $a\geq 0$. The central vertical exact sequence from the following diagram
   \begin{equation}
    \label{diag0}
    \xymatrix@-2ex{
    & & 0 \ar[d] & &\\
    0 \ar[r] & \Oo_H(a) \ar[r]\ar@{=}[d] & \Gg \ar[r]\ar[d] & \Ii_H(1)\cong\Oo_H \ar[r] \ar[d] & 0\\
    0 \ar[r] &  \Oo_H(a)  \ar[r] &  \Ee  \ar[d] \ar[r] &  \Ii_L(1) \ar@{=}[d] \ar[r] & 0\\
               &       & \Oo_H \ar@{=}[r] \ar[d] & \Ii_{L,H}(1)    &  \\
               &        & 0 &   &
    }
  \end{equation}
  \noindent shows that $\Ee$ also fits in (b) of Possible Cases \ref{poss} with $\Gg$ a rank one aCM sheaf. Thus we get the statement.
  \end{proof}

Therefore, the rest of the section will be devoted to study cases (b) and (e).

\begin{lemma}
We have $\Ext_{\PP^{3}}^1(\Oo_X, \Oo_H(a))\cong H^0(\Oo_H(a+2))$ for $a\ge 0$.
\end{lemma}
\begin{proof}
Applying \cite[Lemma 13 in \S 4]{CCM} to $H\subset \PP^3$ with a pair $(\Ff, \Gg)=(\Oo_X, \Oo_H(a))$, we get
$$\Ext_{\PP^3}^1(\Oo_X, \Oo_H(a)) \cong \mathrm{Hom}_H(\mathcal{T}or_1^{\PP^3}(\Oo_X, \Oo_H), \Oo_H(a)),$$
because we have $\Ext_H^1(\Oo_H, \Oo_H(a))=0$ and $\Ext_H^2(\Oo_H, \Oo_H(a))\cong H^0(\Oo_H(-a-3))^\vee=0$. By tensoring (\ref{eqa1}) with $\Oo_H$, we get $\mathcal{T}or_1^{\PP^3}(\Oo_X, \Oo_H) \cong \mathcal{T}or_1^{\PP^3}(\Oo_H(-1), \Oo_H) \cong \Oo_H(-2)$. Thus we get $\Ext_{\PP^3}^1(\Oo_X, \Oo_H(a))\cong H^0(\Oo_H(a+2))$.
\end{proof}

\begin{remark}\label{x+1}
Although we have a non-trivial extension of $\Oo_H(a)$ by $\Oo_X$ as $\Oo_{\PP^3}$-sheaves for $a\ge 0$, it is not an $\Oo_X$-sheaf, because we have $\Ext_{X}^1(\Oo_X,\Oo _H(a)) \cong H^1(\Oo_H(a))= 0$ for all $a\in \ZZ$ by \cite[Proposition III.6.3]{Hartshorne}.
\end{remark}

\begin{proposition}
If $\Ee$ is a simple layered aCM sheaf on $X$, then it is either $\Oo_H (a)$ or $\Oo_X(b)$ for some $a, b\in \ZZ$.
\end{proposition}
\begin{proof}
Let $r\in {\left(\frac{1}{2}\right)}{\ZZ}$ be the rank of $\Ee$. The result is trivial for $r=1/2$. By Theorem \ref{aprop2} it is sufficient to prove that $r\le 1$ (since the ideal sheaf $\Ii_C$ of a plane curve $C\subset H$ is not simple  as it can be easily deduced composing the maps from the exact sequence \ref{eqbbb1} with any non zero morphism $\Oo_H(-d) \to \Oo_H(-1)$).
\\
So assume $r>1$ and fix a filtration with  $0 =\Ee _0\subset \Ee _1\subset \cdots \subset \Ee _{2r-1} \subset \Ee _{2r} =\Ee$ of $\Ee$ with $\Ee _i/\Ee _{i-1} \cong \Oo _H(a_i)$ with $a_i\in \ZZ$ for $i=1,\dots ,2r$. We may assume that $\Ee$ is minimally regular and therefore there is a non-zero map $u: \Ee \rightarrow \Oo _X(1)$. Set $\Ii _A(1):= \mathrm{Im}(u)$. If $a_1>0$, then composing the inclusion $\Ii _A(1) \subseteq \Oo _X(1)$ with the surjection $\Oo _X(1)\rightarrow \Oo _H(1)$ and a non-zero map $\Oo _H(1)\rightarrow \Oo _H(a_1) \subset \Ee$, would imply that $\Ee$ is not simple. Thus we get $a_1\le 0$.

\quad (a) Assume for the moment $a_j>0$ for some $j\ge 2$ and define $s$ to be the minimum among these integers.

\quad {\emph{Claim}}: There is another layering filtration $0 =\Ff _0\subset \Ff _1\subset \Ff _2\subset \cdots \Ff _{2r} =\Ee$ with either (i) $\Ff_1 \cong \Oo _H(a_s)$ or (ii) $a_s=1$, $a_j=0$ for some $j<s$
and $\Ff _2 \cong \Oo _X(1)$.

\quad {\emph{Proof of Claim}}: We use induction on $s$. By assumption $\Ee _s/\Ee _{s-2}$ is an extension of $\Oo _H(a_s)$ by $\Oo _H(a_{s-1})$. If this extension splits, then we may find another layering filtration $0=\Gg_0 \subset \Gg_1 \subset \cdots \subset \Gg_{2r-1} \subset \Gg_{2r}=\Ee$ of $\Ee$ such that
\begin{itemize}
\item $\Gg _i =\Ee _i$ if $i\notin \{s-1,s\}$,
\item $\Gg_{s-1}/\Gg _{s-2} \cong \Oo _H(a_s)$, and
\item $\Gg _s/\Gg _{s-1} \cong \Oo _H(a_{s-1})$.
\end{itemize}
If $s=2$, then we may take $\Ff _i=\Gg _i$ for all $i$. If $s>2$, then we use the inductive hypothesis. In other words, we set $\Gg_{s-1}$ the kernel of the map $\Ee_s \rightarrow \Oo_H(a_{s-1})$ so that $\Gg_{s-1}/\Ee_{s-2} \cong \Oo_H(a_s)$. Actually, in this way only one sheaf in the filtration changes, namely in degree $s-1$, while two maps change; the ones having source and target in degree $s-1$. Now assume that $\Ee _s/\Ee _{s-2}$ is a non-trivial extension of $\Oo _H(a_s)$ by $\Oo _H(a_{s-1})$. Since $a_s>0$ and $a_{s-1} \le 0$, Lemma \ref{ttt} for $X$ and Example \ref{ex1} give $a_s=1$, $a_{s-1}=0$ and $\Ee _s/\Ee _{s-2} \cong \Oo _X(1)$. If $s=2$, then Claim is proved. Now assume
$s>2$. Since $a_j \le 0$ for all $j<0$, we may apply $s-2$ times the twist by $-1$ of Lemma \ref{x+1} to get a new filtration $\Ff _i$ such that $\Ff _2 \cong \Oo _X(1)$, $\Ff _j = \Ee _j$ for all $j\ge s$, and $\Ff _i/\Ff _{i-1} \cong\Oo_H( a_{i-2})$ for $i=2,\dots ,s$.

By Claim we get either a non-zero map $\Ii _A(1) \rightarrow \Ff _1$ or a non-zero map $\Ii _A(1) \rightarrow \Ff _2$. So by composing with $u$ we get that $\Ee$ is not simple, a contradiction.

\quad (b) Assume $a_j\le 0$ for all $j$. Since $\Ee$ is $0$-regular, it is globally generated. In particular, $\Ee /\Ee _{2r-1}$ is globally generated, i.e. $a_{2r}\ge 0$. Our assumption gives $a_{2r}=0$. If $a_1\ge 0$, then we get a non-zero map $\Ee /\Ee _{2r-1} \rightarrow \Ee _1$, which implies that $\Ee$ is not simple. If $a_1<0$, then we get $h^2(\Oo _H(a_1-2)) >0$. This implies $h^2(\Ee(-2))>0$ since $\Ee/\Aa$ is aCM. But it contradicts the $0$-regularity of $\Ee$.
\end{proof}

\begin{lemma}\label{tty}
For a plane curve $C$ of degree $d$ in $X$, we have
\begin{itemize}
\item $\mathrm{ext}_{\PP^3}^1(\Oo_H(a), \Ii_C)={2-a \choose 2} + \max \{2-a-d,0\}$;
\item $\mathrm{ext}_X^1(\Oo_H(a), \Ii_C)\ge {2-a \choose 2}-{2-a-d \choose 2}$,
\end{itemize}
where ${n \choose 2}$ is zero for $n\le 1$.
\end{lemma}
\begin{proof}
Recall that $\Ii_C$ is an extension of $\Oo_H(-d)$ by $\Oo_H(-1)$ with $d=\deg (C)$. Applying the functor $\Hom_{\PP^3}(\Oo_H(a), -)$ to the extension for $\Ii_C$, we get
\begin{equation}\label{yyy2}
\begin{split}
0&\to \Hom_{\PP^3}(\Oo_H(a), \Oo_H(-d)) \to \Ext_{\PP^3}^1(\Oo_H(a), \Oo_H(-1)) \to \Ext_{\PP^3}^1(\Oo_H(a), \Ii_C) \\
&\to \Ext_{\PP^3}^1(\Oo_H(a), \Oo_H(-d)) \to \Ext_{\PP^3}^2(\Oo_H(a), \Oo_H(-1)),
\end{split}
\end{equation}
because we have an isomorphism $\Hom_{\PP^3}(\Oo_H(a), \Oo_H(-1)) \cong \Hom_{\PP^3}(\Oo_H(a), \Ii_C)$, i.e. each morphism $\Oo_H(a) \rightarrow \Ii_C$ factors through $\Oo_H(-1)$. Indeed, we have an injection $\mathrm{Hom}_{\PP^3}(\Oo_H(a), \Oo_H(-1))\rightarrow \Hom_{\PP^3}(\Oo_H(a), \Ii_C)$ and it gives $\mathrm{hom}_{\PP^3}(\Oo_H(a), \Ii_C)\ge {1-a \choose 2}$. On the other hand, applying the functor $\Hom_{\PP^3}(\Oo_H(a), -)$ to (\ref{eq4}), we get the opposite directional inequality, because we have
$$\Hom_{\PP^3}(\Oo_H(a), \Oo_X) \cong \Hom_X(\Oo_H(a), \Oo_X) \cong H^2(\Oo_H(a-2))^\vee \cong H^0(\Oo_H(-1-a))$$
by Serre's duality.

If $a\ge 1$, then we get $a+d\ge 2$. This implies that $\Ext_{\PP^3}^1(\Oo_H(a), \Oo_H(-1)) \cong H^0(\Oo_H(-a))$ and $\Ext_{\PP^3}^1(\Oo_H(a), \Oo_H(-d)) \cong H^0(\Oo_H(1-a-d))$ are trivial by Lemma \ref{ttt}. Thus $\Ext_{\PP^3}^1(\Oo_H(a), \Ii_C)$ is trivial. Now assume $a\le 0$. By Serre's duality and Lemma \ref{ttt}, we get
$$\Ext_{\PP^3}^2(\Oo_H(a), \Oo_H(-1)) \cong \Ext_{\PP^3}^1(\Oo_H(-1), \Oo_H(a-4))^\vee\cong H^0(\Oo_H(a-2))^\vee,$$
which is trivial. Thus the sequence (\ref{yyy2}) becomes
\begin{equation}\label{yyy1}
\begin{split}
0&\to H^0(\Oo_H(-a-d))\to H^0(\Oo_H(-a))\\
& \to \Ext_{\PP^3}^1(\Oo_H(a), \Ii_C) \to H^0(\Oo_H(1-a-d))\to 0.
\end{split}
\end{equation}
If $a+d \ge 2$, then we get $H^0(\Oo_H(-a)) \cong \Ext_{\PP^3}^1(\Oo_H(a), \Ii_C)$, whose dimension is ${2-a \choose 2}$. If $a+d=1$, then similarly we get the dimension ${2-a \choose 2}+1$. Finally if $a+d\le 0$, then each term in (\ref{yyy1}) is non-zero. Thus we get
\begin{align*}
\dim \Ext_{\PP^3}^1(\Oo_H(a), \Ii_C)&=h^0(\Oo_H(-a))+h^0(\Oo_H(1-a-d))-h^0(\Oo_H(-a-d))\\
&=(a^2-5a-2d+6)/2={2-a \choose 2}+(2-a-d)
\end{align*}
and we get the assertion for $\Ext_{\PP^3}^1(\Oo_H(a), \Ii_C)$.

Now consider $\Ext_X^1(\Oo_H(a), \Ii_C)$. From (\ref{yyy2}) with $\PP^3$ replaced by $X$, we get the assertion, due to case $m=1$ of Lemma \ref{ttt}.
\end{proof}


Now recall that an extension of $\Oo_H$ by $\Oo_H$ is isomorphic to either $\Oo_H^{\oplus 2}$ or $\Ii_L(1)$ for a line $L\subset H$. We get that $\mathrm{hom}_X (\Oo _H,\Ii _L(1)) = \mathrm{hom}_X (\Ii _L(1),\Oo _H) =1$ for any line $L\subset H$, which is essentially equivalent to $\Ii _L(1)\not \cong  \Oo _H^{\oplus 2}$ by the proof of Lemma \ref{aaa1}.

\begin{lemma}\label{aaa1}
Let $\Ee$ be a sheaf of rank $3/2$ with the filtration
\begin{equation}\label{aa5}
0=\Ee_0\subset \Ee_1 \subset \Ee_2 \subset \Ee_3=\Ee
\end{equation}
such that $\Ee_i/\Ee_{i-1}\cong \Oo_H$ for all $i\in \{1,2,3\}$. Setting $e_L:= \mathrm{hom}_X(\Oo _H,\Ee )$ and $e_R:= \mathrm{hom}_X(\Ee ,\Oo _H)$, we have the following.
\begin{enumerate}
\item [(i)] $1\le e_L, e_R \le 3$.
\item [(ii)] $e_L=3$ $\Leftrightarrow$ $e_R=3$ $\Leftrightarrow$  $\Ee \cong \Oo _H^{\oplus 3}$.
\item [(iii)] $e_L=2$ (resp. $e_R=2$) if and only if $\Ee$ is an extension of $\Oo _H$ by $\Oo _H^{\oplus 2}$ (resp. $\Oo_H^{\oplus 2}$ by $\Oo_H$).
\item [(iv)] $e_L=e_R=1$ if and only if (\ref{aa5}) is the unique filtration of $\Ee$ with $\Ee _i/\Ee _{i-1}\cong \Oo_H$ for all $i$. In this case $\Ee _2\cong \Ii _L(1)$ for a line $L\subset H$ uniquely determined by $\Ee$.
\item [(v)] $e_L=e_R=2$ if and only if $\Ee \cong \Ii _L(1)\oplus \Oo _H$ for a line $L\subset H$.
\end{enumerate}
\end{lemma}
\begin{proof}
Certainly we have $e_L, e_R \ge 1$. In the exact sequence
\begin{equation}\label{ffe}
0\to \Ee_2 \to \Ee \to \Oo_H \to 0,
\end{equation}
the sheaf $\Ee_2$ is aCM of rank one, admitting an extension of $\Oo_H$ by $\Oo_H$. By the classification of acM sheaf of rank one, we get $\Ee_2 \cong \Oo_H^{\oplus 2}$ or $\Ee_2 \cong \Ii_L(1)$ for a line $L$. In particular, we have $\mathrm{hom}_X(\Oo_H, \Ee_2) \le 2$ and the equality hold only if $\Ee_2 \cong \Oo_H^{\oplus 2}$. Now apply the functor $\Hom_X(\Oo_H, -)$ to (\ref{ffe}) to see that $e_L\le 3$ and the equality hold if and only if $\Ee \cong \Oo_H^{\oplus 3}$. We also obtain similar assertion for $e_R$, by applying the functor $\Hom_X(-,\Oo_H)$.

If $\Ee \cong \Ii _L(1)\oplus \Oo _H$ for some line $L\subset H$, then we have $e_L=e_R=2$, because $\hom_X(\Oo _H,\Ii _L(1)) = \hom_X(\Ii _L(1),\Oo _H) =1$ as mentioned in the paragraph before Lemma \ref{aaa1}. Conversely assume $e_L=e_R=2$, in particular there exist an inclusion $j: \Oo _H^{\oplus 2} \hookrightarrow \Ee$ and a surjection $u: \Ee \twoheadrightarrow \Oo _H^{\oplus 2}$, otherwise the successive quotient $\Ee_i/\Ee_{i-1}$ would have a negative degree. Due to the rank counting, we have $u\circ j \ne 0$  and this gives $\Oo _H\subset j(\Oo _H^{\oplus 2})$ mapped isomorphically by $u$ onto some $\Oo _H\subset \Oo _H^{\oplus 2}$. Hence $\Oo _H$ is a factor of $\Ee$ and we get $\Ee \cong \Ii_L(1)\oplus \Oo_H$ for some line $L\subset H$.
\end{proof}

\begin{remark}
Note that in Lemma \ref{aaa1} $\Ee$ is a semistable sheaf on $\PP^3$ of pure depth two with the same normalized Hilbert polynomial as $\Oo _H$. Since $\Oo _H^{\oplus k}$ is polystable for any positive integer $k$, any nonzero map $u=(u_1, \ldots, u_k): \Oo _H^{\oplus k}\rightarrow \Ee$ has the image isomorphic to $\Oo _H^{\oplus c}$, where $c$ is the dimension of the linear span of $u_i$'s in $\Hom (\Oo _H,\Ee )$. Replacing $e_L$ by $\Hom (\Oo _H^{\oplus k}, \Ee)$, we get in the same way all statements for $e_L$, except (v).
\end{remark}

\begin{remark}\label{aaaa1}
Consider the case (iii) in Lemma \ref{aaa1} with $(e_L, e_R)=(2,1)$. Let $W$ be the set of all $(e_1,e_2)\in \Ext ^1_X(\Oo _H,\Oo _H)^{\oplus 2}$ such that $e_1$ and $e_2$ are linearly independent. In particular, $W$ is an integral variety of dimension $6$. Consider a sheaf $\Ee$ of this type, corresponding to $(e_1, e_2)$. In particular, $\Ee$ has a unique subsheaf $\Ee _2\cong \Oo _H^{\oplus 2}$. Take another sheaf $\Ee'$ corresponding to $(e'_1, e'_2)$, which is isomorphic to $\Ee$. Then $\Ee'$ also has a subsheaf $\Ee '_2\cong \Oo_H^{\oplus 2}$ and we have $f(\Ee_2 )=\Ee'_2$ for any isomorphism $f: \Ee \rightarrow \Ee'$. We get that two sheaves $\Ee$ and $\Ee'$ are isomorphic if and only if there is $M\in \mathrm{GL}(2,\mathbf{k})$ with $M(e_1,e_2) =(e'_1,e'_2)$. So the isomorphism classes are parametrized by the orbits of this action of $\mathrm{GL}(2,\mathbf{k})$ on $W$, i.e. each isomorphism class corresponds to a plane in $\Ext_X^1 (\Oo_H, \Oo_H)\cong \mathbf{k}^3$. So the family of the sheaves of this type is parametrized by $\PP^2$. A similar description may be applied to the case with $(e_L, e_R)=(1,2)$.
\end{remark}

The next goal will be to show the existence of unique non-layered aCM sheaf on $X$ of rank $3/2$ up to twist.

\begin{lemma}
For a point $p\in H\subset X$, we have $\mathrm{ext}_X^1(\Ii_p(1),\Oo_H(-1))=1$.
\end{lemma}
\begin{proof}
We have the standard isomorphism
$$\Ext_X^1(\Ii_p(1),\Oo_H(-1))\cong\Ext_X^1(\Oo_H(-1),\Ii_p(1)\otimes\omega_X))^{\vee}\cong\Ext_X^1(\Oo_H,\Ii_p)^{\vee}.$$
We can apply the functor $\Hom_X(\Oo_H,-)$
  to the short exact sequence
  $$
  0\to\Ii_p\to\Oo_X\to\Oo_p\to 0,
  $$
\noindent to obtain the following strand of the associated long exact sequence:
\begin{align*}
0\cong\Hom_X(\Oo_H,\Oo_X)\to&\Hom_X(\Oo_H,\Oo_p)\cong\mathbf{k}\to\Ext_X^1(\Oo_H,\Ii_p)\\
\to&\Ext_X^1(\Oo_H,\Oo_X)\cong\Ext_X^1(\Oo_X,\Oo_H(-2))^{\vee}\cong 0,
\end{align*}
\noindent concluding the proof.
\end{proof}

\begin{proposition}\label{aprop3}
Let $\Ee_p$ be the unique non-trivial extension of $\Ii_p(1)$ by $\Oo_H(-1)$. Then $\Ee_p$ is a non-layered Ulrich sheaf on $X$ of rank $3/2$. Moreover, for any other point $q\in H$, we have $\Ee_p\cong\Ee_q$.
\end{proposition}

\begin{proof}
Let us consider the unique non-trivial $\Oo_X$-sheaf given as an extension of the form
\begin{equation}\label{ss1}
0\to \Oo_H(-1) \to \Ee_p \stackrel{\pi}{\to} \Ii_p(1) \to 0.
\end{equation}

\quad{\emph{Claim 1}}: $\Ee_p$ is Ulrich.

\quad{\emph{Proof of Claim 1}}: For any $t\in \ZZ$, let $\delta _t: H^1(\Ii _p(t+1)) \rightarrow H^2(\Oo _H(t-1))$ be the coboundary map of the twist by $\Oo _X(t)$ of (\ref{ss1}). From the injection
$$
0\to H^1(\Ee_p(t))\to H^1(\Ii_p(t+1)),
$$
\noindent we get $h^1(\Ee_p(t))=0$ for $t\geq -1$ and $h^1(\Ee _p(t))\le 1$ for all $t\le -2$. We have $h^1(\Ii_p(-1))=h^2(\Oo_H(-3))=1$ and the coboundary map
$$
\delta_{-2} : H^1(\Ii _p(-1)) \to  H^2(\Oo _H(-3))
$$
\noindent corresponds to the non-trivial extension class $[\Ee_p]$. Thus $\delta_{-2}$ is an isomorphism and we get $h^2(\Ee_p(-2))=0$. Assume that $\Ee _p$ is not aCM and let $t_0$ be the largest integer such that $h^1(\Ee _p(t_0)) \ne 0$. We just saw that $t\le -3$. Since $\delta _{-2}\neq 0$, we have $\delta _{t_0+1} \neq 0$. Take an equation $\ell$ of a plane different from $H$. The multiplication by $\ell$ induces a maps between the twist by $\Oo _X(t_0)$ and the twist by $\Oo _X(t_0+1)$ of (\ref{ss1}). The induced map $\alpha : H^1(\Ii _p(t_0+1)) \rightarrow H^1(\Ii _p(t_0+2))$ is an isomorphism. Call $\eta : H^2(\Oo _H(t_0-1)) \rightarrow  H^2(\Oo _H(t_0))$ the map induced by the multiplication by $\ell$. Since $\delta _{t_0+1}\circ \alpha = \eta \circ \delta _{t_0}$, $\delta _{t_0+1}\ne 0$ and $\alpha$ is an isomorphism, we have $\delta _{t_0}\ne 0$, a contradiction. Finally,  the definition of $\Ee_p$ as an extension (\ref{ss1}) gives that $\Ee_p$ has positive depth and that $h^0(\Ee_p(-1))=0$ and $h^0(\Ee_p)=3$. Hence $\Ee_p$ is Ulrich. \qed

Note that we have $\hom_X(\Ee_p, \Oo_X(1))=h^2(\Ee (-3)) =3$ and this gives a two-dimensional projective space $\PP:=\PP \Hom_X(\Ee_p, \Oo_X(1))$ of morphisms $\Ee_p \rightarrow \Oo_X(1)$. If any of such maps is surjective, then its kernel would be isomorphic to $\Oo_H(l)$ for some $l\in \ZZ$ and we would get a different Hilbert polynomial for $\Ee_p$. Thus none of these maps are surjective. Now at least one of these maps is not surjective only at a point (namely, at $p$); in particular this is true for a non-empty subset of $\PP$, because the map $\PP \rightarrow \ZZ$ sending a morphism to the dimension of its zeros is upper semicontinuous. Since we have $\dim \Ext_X^1(\Ii_{p,H}(1), \Oo_X)=1$ but $\mathrm{hom}_X(\Ee, \Oo_X(1))=3$, for each $p\in H$ there is an open neighborhood $U_p$ of $p$ such that for every $q\in U_p$ there is a surjection $\Ee_p \rightarrow \Ii_q(1)$. Thus we get $\Ee _q\cong \Ee _p$ for all $q\in U_p$. Since any two non-empty open subsets of $H$ meet, we get $\Ee _q\cong \Ee _p$ for all $q\in H$. We also see that for every $v\in \PP$ there is $q\in H$ such that $\mathrm{Im}(v) =\Ii _q(1)$. So we get an identification $\PP \cong H$.

\quad{\emph{Claim 2}}: $\Ee_p$ is non-layered.

\quad{\emph{Proof of Claim 2}}:
Assume that $\Ee$ is layered, in particular by Corollary \ref{ff2}, we have a surjection $u: \Ee_p \rightarrow \Oo _H$. Composing with the inclusion $\Oo _H \rightarrow \Oo _X(1)$, we get a morphism $v\in \PP$ such that $\mathrm{Im}(v)\not \cong \Ii _q(1)$ for any $q\in H$, a contradiction. Thus $\Ee_p$ is not layered. \qed

Now Claim 1 and Claim 2 conclude the proof.
\end{proof}

\begin{notation}
Let us denote the unique non-layered minimally regular Ulrich sheaf of rank $3/2$ on $X$ in Proposition \ref{aprop3} by $\overline{\Ee}$. It has the same Hilbert polynomial as $\Oo _H(-1)\oplus \Ii _p(1)$, or as $\Oo _H^{\oplus 3}$.
\end{notation}

\begin{proposition}\label{aprop4}
For any non-layered aCM sheaf $\Ff$ of rank $3/2$ on $X$, $\overline{\Ee} \cong \Ff (t)$ for some $t\in \ZZ$.
\end{proposition}
\begin{proof}
We may assume that $\Ff$ is minimally regular and then $\Ff$ fits in the short exact sequence \ref{eqerr3}:
$$
0 \to \Ll \to \Ee \stackrel{\pi}{\to} \Ii _A(1) \to 0,
$$

\noindent for a closed subscheme $A\subsetneq X$ in Possible Cases \ref{poss}. By the previous discussions we need only to consider cases (b) and (e). Take $A$ as in case (b), i.e. $\Ii _A(1) \cong \Oo _H$. We know by Lemma \ref{kernelacm} that $\Ll$ is aCM of rank one and therefore it is one of the cases described in Theorem \ref{aprop2} and $\Ff$ is layered.


Finally assume case (e). In this case, $\ker(\pi)\cong \Oo_H(a)$ for some $a\in \ZZ$. If $a\le -2$, we get
$$H^1(\Ii_p)=0 \to H^2(\Oo_H(a-1)) \to H^2(\Ff(-1))=0$$
with $h^2(\Oo_H(a-1))>0$, a contradiction. If $a\ge 0$, we get $h^1(\Ff (-2)) \ge h^1(\Ii _p(-1)) -h^2(\Oo _H(a-2)) =1$, a contradiction. Now assume $a =-1$, in particular we get the exact sequence
\begin{equation}\label{ss4}
0\to \Oo_H(-1) \to \Ff \stackrel{\pi}{\to} \Ii_p(1) \to 0.
\end{equation}
Therefore $\Ff$ is the nontrivial extension from Proposition \ref{aprop3}, namely $\Ff\cong\overline{\Ee}$.
\end{proof}

\begin{proposition}\label{jjji}
$\overline{\Ee}$ is a stable $\Oo_X$-sheaf with pure depth two.
\end{proposition}

\begin{proof}
We already showed in Observation at the end of the proof of Proposition \ref{aprop3} that $\overline{\Ee}$ has pure depth two. Moreover, by \cite[Lemma 7.3]{faenzi-pons}, $\overline{\Ee}$ is semistable, and if it was strictly semistable it would fit on an exact sequence of $\Oo _X$-sheaves
\begin{equation}\label{equu1}
0 \to \Aa \stackrel{u}{\to} \overline{\Ee} \to \Bb \to 0
\end{equation}
such that $\Aa$ is Ulrich, $\Bb$ is torsion-free and $\mathrm{rank}(\Aa )+\mathrm{rank}(\Bb )=3/2$.

\quad (a) Assume that $\Aa$ has rank $1/2$. Then, since it is Ulrich, $\Aa\cong\Oo_H$. Let $\pi : \overline{\Ee} \rightarrow \Ii _{p}(1)$ be the surjection in (\ref{ss1}). Since $\hom_X (\Oo _H,\Oo _H(-1)) =0$, the inclusion $u$ in (\ref{equu1}) induces an injective map $\Oo _H\rightarrow \Ii _p(1)$. Since $\Oo _H\cong \mathrm{ker}(f_{w,\Oo _X(1)})$, we have $\hom_X (\Oo _H,\Ii _p(1)) =1$. Thus we get $u(\Oo _H)\subset \pi ^{-1}(\Oo _H)$ and $\pi ^{-1}(\Oo _H)\cong \Oo _H\oplus \Oo _H(-1)$. We see that $\Ee /\Oo _H$ is an extension of $\Ii _{p,H}(1)$ by $\Oo _H(-1)$. By \cite[Lemma 13 in \S 4]{CCM}, we get the first isomorphism of the following
$$\Ext_{\PP^3}^1(\Ii_{p,H}(1), \Oo_H(-1)) \cong  \Ext_H^1(\Ii_{p,H}(1), \Oo_H(-1)) \cong \Ext_X^1(\Ii_{p,H}(1), \Oo_H(-1)),$$
because we have $\Hom _H(\Ii _{p,H},\Oo _H(-1)) =0$. Then the second isomorphism can be induced automatically. Thus $\Ee/\Oo_H$ is isomorphic to either $\Oo _H^{\oplus 2}$ or $\Oo _H(-1)\oplus \Ii _{p,H}(1)$. In the former case, $\Ee$ would be layered, while the latter case is impossible, because $\Ee /\Oo _H$ must be globally generated.

\quad (b) Now assume $\Bb$ has rank $1/2$, in particular we get $\Bb ^{\vee \vee}\cong \Oo _H(b)$ for some $b\in \ZZ$ by Lemma \ref{pop}. Since the map $\Bb \rightarrow \Bb ^{\vee \vee}$ is injective, we get $\Bb \cong \Ii _{Z,H}(b)$ for some zero-dimensional subscheme $Z\subset H$ or $\Bb \cong \Oo_H(b)$. Since $\overline{\Ee}$ is globally generated, $\Bb$ is also globally generated, in particular we get either $b>0$ or $b=0$ and $Z=\emptyset$. The latter case is excluded, because there is no surjection $\overline{\Ee} \rightarrow \Oo _H$ from the argument in the last three lines of the proof of Proposition \ref{aprop3}. The former case is also excluded, because if $b>0$, then the Hilbert polynomial of $\Ii _{Z,H}(b)$ is strictly bigger than the one of $\Oo _H$.
\end{proof}

\begin{corollary}
For each $r\in {\left(\frac{1}{2}\right)}{\ZZ}$ at least $3/2$, there exists a non-layered Ulrich sheaf of rank $r$ on $X$.
\end{corollary}

\begin{proof}
Fix $r\in {\left(\frac{1}{2}\right)}{\ZZ}$ at least $3/2$ and consider the sheaf $\Gg:= \overline{\Ee}\oplus \Oo _H^{\oplus {(2r-3)}}$, for $\overline{\Ee}$ the unique rank $3/2$ non-layered sheaf from the previous remark. $\Gg$ is an Ulrich sheaf of rank $r$. We are going to show that $\Gg$ is  non-layered. Otherwise, by Corollary \ref{ff2}, there exists  a filtration $0 = \Gg _0\subset \Gg_1\subset \dots \subset \Gg _{2r-1}\subset \Gg _{2r} =\Gg$ of $\Gg$ with $\Gg _i/\Gg _{i-1} \cong \Oo _H$ for all $i$.  Consider the composition $u$ of the inclusion $\overline{\Ee} \hookrightarrow \Gg$ with the surjection $\Gg\rightarrow \Gg/\Gg_{2r-1} \cong \Oo_H$. In the end of the proof of Proposition \ref{aprop3}, we proved that $u$ cannot be a surjection. Since $\overline{\Ee}$ is globally generated and $h^0(\Oo_H)=1$, we get that $u$ is a zero map and $\overline{\Ee} \subseteq \Gg_{2r-1}$. By descending induction on $i$, we get $\overline{\Ee} \subseteq \Gg_i$ for all $i$, a contradiction.
\end{proof}

\begin{remark}\label{cases32}
In the rest of the section we will offer a description of indecomposable layered aCM sheaves $\Ee$ on $X$ of rank $3/2$. By the previous discussions we know that such a sheaf $\Ee$ should fit in
\begin{equation}\label{asd}
0\to \Ll \to \Ee \stackrel{\pi}{\to} \Oo_H \to 0,
\end{equation}
where $\Ll$ is aCM and $0$-regular by Lemma \ref{kernelacm} (case (b) of Possible Cases \ref{poss}); by Theorem \ref{aprop2}, $\Ll$ is isomorphic to either $\Oo_X(a)$, with $a\in\ZZ$; or $\Ii_C(a)$ for a plane curve $C\subset H$ with $a\ge d$; or $\Oo_H(a)\oplus \Oo_H(b)$, with $a\geq b\geq 0$. The first case is excluded:  $\Ext_X^1(\Oo_H, \Oo_X(t)) \cong H^1(\Oo_H(-t-2))^\vee$ is trivial and $\Ee$ would be decomposable. In Example \ref{bbb11} and Lemma \ref{bbb1} below, we describe the latter case. It turns out in the proof of Lemma \ref{wwe} that such sheaves fall into the case (2-iii) of Theorem \ref{iii}. 
\end{remark}

\begin{example}\label{bbb11}
For fixed nonnegative integers $a\ge b \ge 0$, set $d:= a+b$ and let $\Aa (b,d)$ be the set of isomorphism classes of sheaves fitting in an exact sequence
\begin{equation}\label{eqbbb2}
0 \to \Oo _H(a)\oplus \Oo _H(b)\stackrel{i}{\to} \Ee \stackrel{\pi}{\to} \Oo _H\to 0.
\end{equation}
$\Aa (b,d)$ is parametrized, not necessarily finite-to-one, by a vector space
$$\mathbf{E}(b,d):=\Ext ^1_X(\Oo_H, \Oo_H(a))\oplus \Ext_X^1(\Oo_H, \Oo_H(b))$$
of dimension ${a+3 \choose 2}+{b+3 \choose 2}$. Every element in $\Aa(b,d)$ is aCM, because it is an extension of aCM sheaves. Fix $[\Ee] \in \Aa (b,d)$. Since the case $(a,b)=(0,0)$ is already described in Lemma \ref{aaa1} and Remark \ref{aaaa1}, we assume that $a>0$; so we also get $d>0$. The sequence (\ref{eqbbb2}) gives $H^2(\Ee (-2)) =0$ and $H^2(\Ee (-3)) \ne 0$. In particular, $\Ee$ is minimally regular.

Suppose that $\Ee$ is induced by $\epsilon \in \mathbf{E}(b,d)$ and write $\epsilon = (e_1 ,e_2)$ with $e_1 \in  \Ext ^1_X(\Oo_H, \Oo _H(a))$ and $e_2 \in \Ext ^1_X(\Oo_H, \Oo _H(b))$. If $e_1 =0$ (resp. $e_2 =0$), then $\Oo_H(a)$ (resp. $\Oo_H(b)$) is a factor of $\Ee$. Now assume $e_1 \ne 0$ and $e_2\ne 0$. The extension $e_1 $ (resp. $e_2$) induces a rank one aCM sheaf $\Ii _C(a+1)$ (resp. $\Ii _D(b+1)$) for uniquely determined plane curves $C, D\subset H$ with $\deg (C)=a+1$ and $\deg (D)=b+1$. Conversely, the curves $C$ and $D$ determine $\alpha$ and $\beta$, respectively, up to a constant, but the constants may be different; for two nonzero constants $c$ and $c'$, we may consider the automorphism of $\Oo _H(a)\oplus \Oo _H(b)$ obtained by the multiplication by $c$ in the first factor and by the multiplication by $c'$ in the second factor, to show that the sheaf induced by $(ce_1, c'e_2)$ is isomorphic to $\Ee$.

Assume $b>0$, and in particular $\Oo _H(a)\oplus \Oo _H(b)$ is uniquely determined by the Harder-Narasimhan filtration of $\Ee$. Thus $\epsilon$ is uniquely determined by the isomorphism class of $\Ee$, up to isomorphisms of $\Oo _H(a)\oplus \Oo_H(b)$. If $\Oo _H(k)$ is a factor of $\Ee$ with $k>0$, then it is a factor of $\Oo _H(a)\oplus \Oo _H(b)$, because $\pi (\Oo _H(k)) =0$ and there is a surjection $\Oo _H(a)\oplus \Oo _H(b)\rightarrow \Oo _H(k)$ only if either $k=a$ or $k=b$. If $\Oo_H$ is a factor of
$\Ee$, then we get $\Ee\cong \Oo _H\oplus \Oo _H(a)\oplus \Oo _H(b)$, because $\hom_X (\Oo _H(a)\oplus \Oo _H(b),\Oo _H)=0$ and we have the surjection $\pi$.
\end{example}

\begin{lemma}\label{bbb1}
Let $\Bb (b,d)\subset \Aa(b,d)$ consist of all indecomposable $[\Ee] \in \Aa (b,d)$. Then $\Aa (b,d)\setminus \Bb (b,d)$ is parametrized by at most $d+1$ proper linear subspaces of $\mathbf{E}(b,d)$. in particular, $\Bb(b,d)$ is not empty.
\end{lemma}

\begin{proof}
Fix $[\Ee] \in \Aa (b,d)\setminus \Bb (b,d)$. Since $\Ee$ has rank $3/2$ and it is decomposable, it has a factor of rank $1/2$. Let $k$ be the minimum among the integers such that $\Ee$ has a factor $\Oo _H(k)$, i.e. $\Ee \cong \Oo _H(k)\oplus \Ff$ with $\Ff$ a rank one aCM sheaf.

First assume $k>0$. Since $\pi (\Oo _H(k)) =0$, we get that $\Ff$ fits into an exact sequence
$$0 \rightarrow \Oo _H(d-k) \rightarrow \Ff \rightarrow \Oo _H \rightarrow 0.$$
So $\Ff$ is isomorphic to either $\Oo _H\oplus \Oo _H(d-k)$ or $\Ii _E(d-k+1)$ for some plane curve $E\subset H$ with $\deg (E) =d-k+1$. The former case is impossible due to the definition of $k$. In the latter case, the extension classes arising as $\Ff$ are parameterized by $\Ext_X^1(\Oo_H, \Oo_H(d-k)) \cong H^0(\Oo_H(d-k+1))$, whose dimension is ${d-k+3\choose 2}$. So if $k\ge b$, then the isomorphism classes arising as $\Ee$ are parameterized by a proper linear subspace of $\mathbf{E}(b,d)$. If $k<b$, then each map $\mathrm{Hom}(\Oo_H(a)\oplus \Oo_H(b), \Ee)$ has image contained in $\Ff$. So we get an injective map $\Oo_H(a)\oplus \Oo_H(b) \rightarrow \Ii_E(d-k+1)$. The composition of this injective map with the surjection $\Ii_E(d-k+1) \rightarrow \Oo_H$ cannot be trivial, otherwise we would get an injection $\Oo_H(a)\oplus \Oo_H(b) \rightarrow \Oo_H(d-k)$. Thus we have $b=0$, i.e. $(a,b)=(d,0)$. In this case we see that the extension classes arising as $\Ff$ are parameterized by $\Ext_X^1(\Oo_H, \Oo_H(d))$ whose dimension is $d+3 \choose 2$. So such sheaves are parameterized by a proper linear subspace of $\mathbf{E}(0,d)$.

Now assume $k=0$. So we get $\Ee \cong \Oo _H\oplus \Ff$ with either
\begin{itemize}
\item $\Ff \cong \Oo _H(i)\oplus \Oo _H(d-i)$ for some integer $i$ with $0\le 2i \le d$
,or
\item $\Ff \cong \Ii _D(j)$ for some curve $D\subset H$ with $z:= \deg (D)=2j-d-1$.
\end{itemize}
Note that $\Ff$ cannot be a line bundle on $X$ by the regularity condition of $\Ee$. In the former case, we get $i=0$; $\Ff \cong \Oo_H\oplus \Oo_H(d)$. This implies that $\Ee \cong \Oo_H^{\oplus 2} \oplus \Oo_H(d)$, which corresponds to the trivial element of $\mathbf{E}(0,d)$. Now assume the latter case. Since $\Ee$ is $0$-regular, $\Ff$ is globally generated. In particular, we get $j\ge z$. If $j>z$, then (\ref{eqbbb1}) gives $\hom_X (\Ii _D(j),\Oo _H)=0$. So the factor $\Oo _H$ of $\Ee$ induces the splitting of (\ref{eqbbb2}), in particular we get $\Ff \cong \Oo_H(a)\oplus \Oo_H(b)$, a contradiction. Now assume $j=z$, i.e. $j=d+1$. If $b>0$, then the map $i$ from the exact sequence (\ref{eqbbb2}) has image contained in $\Ff$, while (\ref{eqbbb2}) implies that each map in $\mathrm{Hom} (\Oo _H(a)\oplus \Oo _H(b),\Ii _D(j))$ has image contained in $\Oo _H(d)$. So $i$ can not be injective, a contradiction. In case $b=0$, i.e. $(a,b)=(d,0)$, we get a proper linear subspace of $\mathbf{E}(0,d)$ as above.
\end{proof}

\begin{remark}
Fix $\Ee \in \Bb (b,d)$ for some integers $b, d$ and assume $b>0$. For any two positive integers $a'$ and $b'$, we have $\Hom(\Oo _H(a')\oplus \Oo _H(b'),\Oo _H)=0$. So if $\{a,b\} \ne \{a',b'\}$, there is no injective map $\Oo _H(a')\oplus \Oo _H(b') \rightarrow \Ee$. In particular, we have $\Bb (b,d)\cap \Bb (b',d') =\emptyset$, if $(b,d)\ne (b',d')$.
\end{remark}

\begin{lemma}\label{wwe}
Let $\Ee$ be a minimally regular indecomposable sheaf fitting into (\ref{asd}). If $\Ee$ is not as in Lemma \ref{aaa1}, then there exists a plane curve $C\subset H$ of degree $d$ such that $\Ee$ fits into one of the following sequences:
\begin{enumerate}
\item $0\rightarrow \Ii_C(a) \rightarrow \Ee \rightarrow \Oo_H \rightarrow 0$ for $a\ge d$;
\item $0\rightarrow \Oo_H(b) \rightarrow \Ee \rightarrow \Ii_C(d) \rightarrow 0$ for $0\le b<d$;
\end{enumerate}
\end{lemma}
\begin{proof}
By Remark \ref{cases32}, we can assume that $\Ll \cong \Oo _H(a)\oplus \Oo _H(b)$ with $a\ge b\ge 0$. The quotient sheaf $\Ee/\Oo_H(b)$ is an extension of $\Oo_H$ by $\Oo_H(a)$. So it is isomorphic to either $\Oo_H(a)\oplus \Oo_H$ or $\Ii_C(a+1)$ for a plane curve $C\subset H$ of degree $a+1$ by Example \ref{ex1}. The latter case falls into case (2).

In the former case, let $u: \Ee \rightarrow \Oo_H(a)\oplus \Oo_H$ be the quotient map and set $\Ff:=u^{-1}(\Oo_H(a))$. Then $\Ff$ is an extension of $\Oo_H(a)$ by $\Oo_H(b)$ and $\Ee$ is an extension of $\Ff$ by $\Oo_H$. By Example \ref{ex1} $\Ff$ is isomorphic to either
$$\Oo_H(a)\oplus \Oo_H(b) \text{ or } \Ii_D(b+1)$$
for a plane curve $D$ of degree $b-a+1$. Note that the plane curve $D$ makes sense only if $b-a+1\ge 0$, i.e. $b\in \{a-1, a\}$, and if $b=a-1$, we have $\Ii_D(2b+2-a) \cong \Oo_X(a)$.
Assume first that $\Ff \cong \Oo_H(a)\oplus \Oo_H(b)$. In particular, $\Ee$ is induced by
$$(e_1, e_2) \in \Ext_X^1(\Oo_H(a), \Oo_H)\oplus \Ext_X^1(\Oo_H(b), \Oo_H).$$
From the indecomposability of $\Ee$, we get $a\le 1$. If $a=1$, then the sheaf corresponding to $e_1$ is isomorphic to $\Oo_X(1)$ and $\Ee$ is an extension of $\Oo_X(1)$ by $\Oo_H(b)$, which is trivial by Remark \ref{x+1}. So we have $(a,b)=(0,0)$ and we already describe the case in Lemma \ref{aaa1} and Remark \ref{aaaa1}. Now assume that $\Ff \cong \Ii_D(b+1)$, where $D$ is either empty or a line. If $b=a-1$, then we get $\Ff \cong \Oo_X(a)$. In this case $\Ee$ is decomposable by Remark \ref{x+1}. So we may assume that $\Ff \cong \Ii_L(a+1)$. In particular, $\Ee$ is an extension of $\Ii_L(a+1)$ by $\Oo_H$. Apply the functor $\mathrm{Hom}_X(- , \Oo_H)$ to the exact sequence
$$0\to \Oo_H(a) \to \Ii_L (a+1) \to \Oo_H(a) \to 0,$$
we get $\Ext_X^1(\Ii_L(a+1), \Oo_H)=0$ for $a\ge 2$ by Example \ref{ex1}. The case $a=0$ falls into case (2) with $(b,d)=(0,1)$, while in the case $a=1$ we have $\mathrm{hom}_X(\Ff, \Oo_H)=0$. So the map $\pi$ gives the splitting of (\ref{asd}).
\end{proof}

\begin{proof}[Proof of Theorem \ref{iii}:]
If $\Ee$ is not layered, then we may use Propositions \ref{aprop3} and \ref{jjji}. If $\Ee$ is decomposable, then we use Theorem \ref{aprop2}. Finally when $\Ee$ is indecomposable and layered, we may use the proof of Proposition \ref{aprop3} and Lemma \ref{wwe}.
\end{proof}

From the proof of Lemma \ref{wwe}, the case (2) of Lemma \ref{wwe} occurs when $\Ee$ fits into the exact sequence (\ref{eqbbb2}), which is already described in Example \ref{bbb11} and Lemma \ref{bbb1}.

Now assume $\mathrm{ker}(\pi) \cong \Ii _C(a)$ for some degree $d$ curve $C$ and a unique integer $a$ as in (1) of Lemma \ref{wwe}. Indeed, if $a\ge d$, then we get $\dim \Ext_X^1(\Oo_H, \Ii_C(a))\ge {2+a \choose 2}-{2+a-d \choose 2}$, which is positive.

\begin{lemma}\label{iid}
Any non-trivial extension of $\Oo_H$ by $\Ii_C(a)$ with $a\ge d$ is indecomposable.
\end{lemma}
\begin{proof}
Let $\epsilon$ be the non-trivial extension class of $\Oo_H$ by $\Ii_C(a)$ corresponding to $\Ee$ and assume that $\Ee$ is decomposable with $\Ee \cong \Oo _H(b)\oplus \Gg$ for some aCM sheaf $\Gg$ of rank $1$, where $\Gg$ is described in Theorem \ref{aprop2}. Since $\Ee$ is globally generated, $\Gg$ is also globally generated and $b\ge 0$.

\quad (i) First assume $b>0$, in particular we have $\Hom_X(\Oo _H(b),\Oo _H)=0$ with a surjective map $\Gg \rightarrow \Oo _H$. Since $\Gg$ is globally generated, by Theorem \ref{aprop2}, $\Gg$ is isomorphic to either $\Oo _X$, $\Oo _H^{\oplus 2}$ or $\Ii _L(1)$ for some line $L\subset H$.
If $\Gg \cong \Oo _X$, then we get
\begin{equation}\label{o1o}
1 + \binom{a+1}{2} +\binom{a-d+2}{2} = h^0(\Ee )=1+ \binom{b+2}{2}.
\end{equation}
Since $\Hom_X (\Oo _H(b),\Oo _H)=0$, we get an injective map $j: \Oo _H(b)\rightarrow \Ii _C(d)$ with the cokernel isomorphic to $\Oo_H(-1)$, which is the kernel of the surjection $\Gg \rightarrow \Oo_H$. Since $\Ii _C(a)$ is an extension of $\Oo _H(a-d)$ by $\Oo _H(a-1)$, we get $b=a-1$, which is impossible by (\ref{o1o}).

Thus we get either  $\Gg \cong \Oo _H^{\oplus 2}$ or $\Gg \cong  \Ii _L(1)$ for some line $L\subset H$ with the equality
\begin{equation}\label{o1o1}
1 + \binom{a+1}{2} +\binom{a-d+2}{2} = \binom{b+2}{2} +2
\end{equation}
as above. The nonzero map $j: \Oo _H(b)\rightarrow \Ii _C(d)$ gives that either $b=a-1$ or $b\le a-d$, and this implies from (\ref{o1o1}) that $a=1$. Then we get $d=1$ and $b=0$, a contradiction.

\quad (ii) Now assume $b=0$, in particular the map $\Ee \rightarrow \Oo _H$ induces the zero-map $\Oo _H\rightarrow \Oo _H$, because $\epsilon \ne 0$ and any non-zero map $\Oo _H\rightarrow \Oo _H$
is an isomorphism. Thus we get a surjective map $\Gg \rightarrow \Oo _H$. So as before $\Gg$ is isomorphic to either $\Oo _X$, $\Oo _H^{\oplus 2}$ or $\Ii _L(1)$ for some line $L\subset H$. If  $\Gg \cong \Oo _X$, then $h^0(\Ee )=2$ and $h^0(\Ii _C(a)) =1$, a contradiction. If either $\Gg \cong \Oo _H^{\oplus 2}$ or $\Gg \cong  \Ii _L(1)$, then we get $h^0(\Ee )=3$ and $h^0(\Ii _C(a)) =2$, i.e. $a=d=1$. Thus we are in the set-up of Lemma \ref{aaa1}, and it is indecomposable.
\end{proof}

\noindent The sheaf $\Ii _C(a)$ is an extension of $\Oo _H(a-d)$ by $\Oo _H(a-1)$. First assume $a\ge d+1$. Since $\omega _C \cong \Oo _C(d-3)$, the exact sequence
$$0\to \Ii _C\to \Oo _X\to \Oo _C\to 0$$
gives $h^2(\Ii _C(a-3)) =0$. Therefore $h^2(\Ee (-3)) =1$ and by Serre's duality there is a unique surjection $\Ee \rightarrow \Oo_H$, up to a scalar, i.e. a unique subsheaf isomorphic to $\Ii_C(a)$ for some integer $a$ and some curve $C$. Hence if $f: \Ee \rightarrow \Ee'$ is an isomorphism, then $\Ee'$ contains $\Ii_C(a)$ and $f(\Ii_C(a)) =\Ii_C(a)$. We get that the isomorphism classes of $\Ee$ are parametrized by the quotient of the family of the non-zero extensions of $\Oo_H$ by $\Ii_C(a)$, by the action of $\mathrm{Aut} (\Ii_C(a))$. Note that by Proposition \ref{ggg} below we have $\dim \mathrm{Aut} (\Ii_C(a)) = 1+\binom{d+1}{2}$. In other words, we have the action of the algebraic group $\mathrm{Aut} (\Ii_C(a))$ on the quasi-affine integral variety $\Ext_X^1(\Oo_H, \Ii_C(a))\setminus \{0\}$ so that the isomorphism classes of these sheaves are parametrized by the orbits of the algebraic group $\mathrm{Aut} (\Ii_C(a))$.

\begin{proposition}\label{ggg}
For a plane curve $C\subset H$ of degree $d$, the group $\Aut (\Ii _C)$ is identified with the group of matrices
$$\left \{\binom{a \quad f}{0 \quad a^{-1}}~ \bigg|~ a\in\mathbf{k}\setminus \{0\}~,~f \in\mathbf{k}[x,y,z]_{d-1} \right \}.$$
In particular, $\dim\mathrm{End} (\Ii _C) = 1+\binom{d+1}{2}$.
\end{proposition}
\begin{proof}
Automorphisms of $\Ii_C$ extend to automorphisms of its $\Oo_{\mathbb{P}^3}$-resolution
\[
0\to \Oo_{\PP^3}(-2)\oplus\Oo_{\PP^3}(-d-1)\stackrel{A}{\to}\Oo_{\PP^3}(-1)\oplus\Oo_{\PP^3}(-d)\to \Ii_C\to 0,
\]
\noindent namely, to pairs of matrices
$$(M,N)\in \mathrm{End}(\Oo_{\PP^3}(-2)\oplus\Oo_{\PP^3}(-d-1))\oplus \mathrm{End}(\Oo_{\PP^3}(-1)\oplus\Oo_{\PP^3}(-d))$$
such that $N^{-1}AM=A$. Developing this equation, the pairs correspond to matrices as on the statement. Now, since $\Aut (\Ii _C)$ is a non-empty open subset of $\End (\Ii _C)$,we get the second part of the statement as well.
\end{proof}

There remains the case $a=d$. If $h^2(\Ee (-3)) =1$, then again $\Ee$ fits in a unique extension of $\Oo _H$ by $\Ii _C(d)$. So we have the same description as in the case $a\ge d+1$. Now assume $h^2(\Ee (-3)) >1$. If $a=d=1$, then we are in the set up of Lemma \ref{aaa1} with $e_R\ge 2$. Now assume $a=d>1$. Note that $\Ee$ contains a unique copy of $\Oo _H(d-1)$ with $\Ff:= \Ee /\Oo _H(d-1)$, where $\Ff$ is isomorphic to $\Oo_H^{\oplus 2}$ or $\Ii _L(1)$ for some line $L\subset H$. Thus any isomorphism $f: \Ee \rightarrow \Ee'$ sends $\Oo_H(d-1)$ to the corresponding copy of $\Oo_H(d-1)$ of $\Ee'$. So $f$ induces an isomorphism of extensions of $\Ff$ by $\Oo _H(d-1)$. We get that the isomorphism classes of $\Ee$ are parametrized by the quotient of the family of non-zero extensions of $\Ff$ by $\Oo _H(d-1)$, by the action of $\mathrm{Aut}(\Ff )$. See Lemma \ref{ggg} for the computation of $\mathrm{Aut}(\Ff )$ in the case $\Ff \cong \Ii _L(1)$. In all cases $\Ee$ is decomposable only if either
\begin{itemize}
\item $\Ee \cong \Oo _H(d-1)\oplus \Ff$, which is possible only if the extension of $\Ff$ by $\Oo _H(d-1)$ is the trivial one, or
\item $\Ee \cong \Oo _H\oplus \Gg$ for some aCM sheaf $\Gg$ that is $0$-regular with $h^0(\Gg ) = \binom{d+1}{2} +1$.
\end{itemize}
If the inclusion $j: \Oo _H\rightarrow \Ee$ induced from an $\Oo_H$-factor of $\Ee$, does not split the extension $\Ee$ of $\Oo _H$ by $\Ii _C(d)$, then the surjection $\Ee \rightarrow \Oo_H$ maps $j(\Oo _H)$ onto zero, in particular $j(\Oo _H)$ is a saturated subsheaf of $\Ii _C(d)$, i.e. the quotient is torsion-free. Its existence would imply that $\Ii _C(d)\cong \Oo _H\oplus \Oo _H(d-1)$, contradicting the fact that $\Ii _C(d)$ is not an $\Oo _H$-sheaf.

\begin{remark}
Among the aCM sheaves of rank $3/2$ studied in this section, only the sheaves of the form $\Oo_H(a)\oplus \Oo_H(b)\oplus \Oo_H(c)$ are $\Oo_H$-sheaves; for any other sheaf $\Ee$ there is no closed  proper subscheme $Y\subset X$ such that $\Ee$ is an $\Oo_Y$-sheaf.
\end{remark}


\section{Ulrich sheaves}
In this section, we discuss the (non)-existence of Ulrich sheaves on $X$. Recall that $\Delta$ is the collection of aCM vector bundles, admitting an extension of $\Ss_H$ by $\Ss_H(-1)$; see Example \ref{qwe}.

\begin{lemma}\label{yyy}
For every $[\Ee] \in \Delta$, its twist $\Ee(1)$ is Ulrich.
\end{lemma}

\begin{proof}
By definition of $\Delta$, the vector bundle $\Ee$ is aCM with $\Ee _{|H} \cong \Omega _H^1(1)\oplus \Oo _H\oplus \Oo _H(-1)$. So it fits into the exact sequence
\begin{equation}\label{eqq1}
0 \to \Omega _H^1\oplus \Oo _H(-1)\oplus \Oo _H(-2)\to \Ee \to \Omega _H^1(1)\oplus \Oo _H\oplus \Oo _H(-1) \to 0.
\end{equation}
In particular, we have $h^0(\Ee )\le 1$ and $H^0(\Ee)$ is the kernel of the coboundary map $\delta : H^0(\Oo _X)\rightarrow H^1(\Omega _H^1)$. The latter cohomology group is one-dimensional and $\delta$ must be an isomorphism, because $H^1(\Ee )$ is trivial. Again by (\ref{eqq1}) we get
\begin{align*}
h^0(\Ee (1)) &=h^0(\Oo _H)+h^0(\Oo _H)+h^0(\Oo _H(1)) +h^0(\Omega _H^1(2)) \\
&=1+1+3+3=8.
\end{align*}
Since $\Ee(1)$ is an initialized vector bundle of rank four on $X$ with degree two, it is Ulrich.
\end{proof}

\begin{lemma}\label{errb3}
For each $[\Ee] \in \Delta$, we have $4 \le \dim \End (\Ee ) \le 9$.
\end{lemma}

\begin{proof}
From the Euler sequence
$$0 \to \Omega _H^1 \to \Oo _H(-1)^{\oplus 3}\to \Oo _H\to 0,$$
we get $h^1(\Omega _H^1\otimes (\Omega _H^1)^\vee (-1)) = h^0((\Omega _H^1)^\vee (-1)) =3$. Set
$$\Aa:= \Oo _H ^{\oplus 2}\oplus \Oo _H(-1)\oplus \Omega _H^1 \otimes (\Omega _H^1)^\vee \oplus \Omega _H^1(1) \oplus \Omega _H^1(2) \oplus (\Omega _H^1)^\vee (-1) \oplus (\Omega _H^1)^\vee (-2).$$
We have $h^0(\Aa (-1)) =0$, $h^0(\Aa ) = 9$ and $h^1(\Aa (-1)) =2 +h^1(\Omega _H^1\otimes (\Omega _H^1)^\vee (-1))=5$. Since $\mathcal{E}nd (\Ee)$ is locally free, we have an exact sequence
\begin{equation}\label{eqerrb2.2}
0 \to \Aa (-1) \to \mathcal{E}nd (\Ee )\to \Aa \to 0,
\end{equation}
in particular we get $4 \le \dim \End (\Ee )\le 9$.
\end{proof}


\begin{proposition}\label{ff1}
Let $\Ee$ be a layered sheaf of rank $r\ge 1$ on $X$ with the filtration in Definition \ref{llay}. Set $\mathrm{ind}(\Ee):= \max\{t\in \ZZ \mid H^0(\Ee (-t)) \ne 0\}$. Then we have
\begin{itemize}
\item [(i)] $\mathrm{ind}(\Ee)= \max \{a_1,\dots ,a_{2r}\}$;
\item [(ii)] $h^0(\Ee (-\mathrm{ind}(\Ee))) = \sharp \{i\in \{1,\dots, 2r\}\mid a_i =\mathrm{ind}(\Ee)\}$.
\end{itemize}
\end{proposition}

\begin{proof}
If we let $\rho= \max \{a_1,\dots ,a_{2r}\}$, then the filtration of $\Ee$ gives $H^0(\Ee (-t)) =0$ if $t >a_i$ for all $i$. So we get $\mathrm{ind}(\Ee) \le \rho$. If $\rho=a_1$, then we get $H^0(\Ee (-a_1)) \ne 0$ and $\mathrm{ind}(\Ee) \ge a_1$. Similarly, if $\rho=a_i$ for $i\ge 2$, we get $H^0((\Ee /\Ee _{i-1})(-\rho)) \ne 0$. Since $\Ee _{i-1}$ is aCM, we get $H^0(\Ee (-\rho)) \ne 0$ and $\mathrm{ind}(\Ee)\ge \rho$.
\end{proof}

\begin{corollary}\label{ff2}
If $\Ee$ is a layered Ulrich sheaf of rank $r$ with the filtration in Definition \ref{llay}, then we have $a_i=0$ for all $i$.
\end{corollary}

\begin{proof}
By Lemma \ref{ff1}, we have $0=\mathrm{ind}(\Ee) = \max \{a_1,\dots ,a_{2r}\}$ and $a_i =0$ for $2r$ indices $i$, concluding the proof.
\end{proof}

\begin{proposition}\label{ff2.1-}
There is no layered Ulrich vector bundle on $X$.
\end{proposition}
\begin{proof}
Let us suppose that there exists an Ulrich vector bundle $\Ee$ of rank $r$ on $X$ with filtration:
\begin{equation}\label{ulrfilt}
 0 \to \Ee _{i-1}\to \Ee_{i} \to \Oo _H\to 0,
\end{equation}
\noindent for $i=2,\dots,2r$ with $\Ee_1\cong\Oo_H$ and $\Ee_{2r}\cong\Ee$.

\quad \emph{Claim 1:} $H^j(\Ee_{i|H}(-1))=0$ for $i=1,\dots,2r$ and $j\in\{0,1,2\}$.

\quad \emph{Proof of Claim 1:} We are going to prove the claim by induction on $i$; notice that the claim is true for $i=1$. On the other hand, tensoring the short exact sequence (\ref{eqa1}) by the $\Oo_X$-sheaf $\Ee_i$ we obtain
$$
0\to\mathcal{T}or^1_X(\Ee_i,\Oo_H)\to\Ee_{i|H}(-1)\to\Ee_i\to\Ee_{i|H}\to 0,
$$
\noindent so we can deduce $\mathcal{T}or^1_X(\Oo_H,\Oo_H)\cong\Oo_H(-1)$ for $i=1$ and $\dim\supp \mathcal{T}or^1_X(\Ee_i,\Oo_H)\leq 1$ for $i\geq 2$. Tensoring the short exact sequence
(\ref{ulrfilt}) by $\Oo_H(-1)$, we get
$$
\mathcal{T}or^1_X(\Ee_i,\Oo_H)(-1)\stackrel{\phi_3}{\to}\Oo_H(-2)\stackrel{\phi_2}{\to} \Ee _{i-1 |H}(-1)\stackrel{\phi_1}{\to} \Ee_{i |H}(-1) \to \Oo _H(-1)\to 0.
$$
\noindent Splitting the previous exact sequence into short ones, we can see first that the surjection
$$
H^2(\Ee_{i-1|H}(-1))\twoheadrightarrow H^2(\Image \phi_1)\cong H^2(\Ee_{i|H}(-1))
$$
\noindent gives us $H^2(\Ee_{i|H}(-1))=0$ by the induction's hypothesis. Next
$$
H^1(\Ee_{i-1|H}(-1))\cong H^1(\Image \phi_1)\cong H^2(\Image \phi_2)
$$
\noindent and the last group is zero due to the existence of the surjection $0=H^2(\Oo_H(-2))\twoheadrightarrow H^2(\Image \phi_2)$. Finally, we have the chain of equalities
$$
H^0(\Ee_{i|H}(-1))=H^0(\Image \phi_1)=H^1(\Image \phi_2)=H^2(\Image \phi_3)
$$
\noindent and again the last cohomology is zero due to the surjection $H^2(\mathcal{T}or^1_X(\Ee_i,\Oo_H)(-1))\twoheadrightarrow H^2(\Image \phi_3)$ and that $H^2(\mathcal{T}or^1_X(\Ee_i,\Oo_H)(-1))=0$ since the support of this sheaf is at most one-dimensional. \qed

\quad \emph{Claim 2:} $H^0(\Ee^{\vee}_{|H})=0$.

\quad \emph{Proof of Claim 2:}
After tensoring the short exact sequence (\ref{eqa1}) by $\Ee^{\vee}$ (notice that, since $\Ee$ is a vector bundle, the operations of dualizing and restricting to $H$ do commute), we get
\begin{equation}\label{eqa6}
0\to\Ee^{\vee}(-1)_{|H}\to\Ee^{\vee}\to\Ee^{\vee}_{|H}\to 0.
\end{equation}

\noindent Since $h^0(\Ee^{\vee})=h^1(\Ee^{\vee})=h^2(\Ee^{\vee}(-1))=0$ by  Serre duality and the fact of $\Ee$ being Ulrich, we deduce from the long exact sequence of cohomology groups associated to (\ref{eqa6})
$$
H^0(\Ee_{|H}^{\vee})\cong H^1(\Ee_{|H}^{\vee}(-1))\cong H^2(\Ee_{|H}^{\vee}(-2))\cong H^0(\Ee_{|H}(-1))=0,
$$
\noindent by Claim 1, where the last isomorphism is obtained applying Serre duality on $H$. This concludes the proof of Claim 2. \qed

Finally, after tensoring (\ref{ulrfilt}) by $\Ee^{\vee}$ for any $i=2,\dots,2r$ and using Claim 1, we would obtain
\[
h^0(\Ee\otimes\Ee^{\vee})=\dots=h^0(\Ee_i\otimes\Ee^{\vee})=\dots=h^0(\Ee^{\vee}_{|H})=0,
\]
a contradiction.
\end{proof}

\begin{proposition}\label{u1}
For any $[\Ee]\in \Delta$ in Example \ref{qwe}, $\Ee$ is not layered.
\end{proposition}

\begin{proof}
Since $\Ee(1)$ is Ulrich for each $[\Ee]\in \Delta$, the result is immediate from Proposition \ref{ff2.1-}.
\end{proof}

\begin{theorem}\label{tthhmm}
For each $r\in {\left(\frac{1}{2}\right)}{\ZZ}_{>0}$, there exists a layered indecomposable Ulrich sheaf with rank $r$.
\end{theorem}

\begin{proof}
By Corollary \ref{ff2} we need only to check the indecomposability of some layered sheaf $\Ee$ with a filtration $0 =\Ee _0\subset \cdots \subset \Ee _{2r} = \Ee$ with $\Ee _i/\Ee _{i-1} \cong \Oo _H$ for all $i$. The case $r =1/2$ is trivial. Note that the assertion is also true for $r \in \left\{1, \frac{3}{2}\right\}$; for $r=1$, up to isomorphism it only gives the indecomposable sheaf $\Ii _L(1)$ with $L\subset H$ a line, and the case $r =\frac{3}{2}$ comes as a particular case of Lemma \ref{iid} with $(a,d)=(1,1)$.

Set $\Ee _1:= \Oo _H$. For each $r\in {\left(\frac{1}{2}\right)}{\ZZ}_{\geq 1}$, let $\Ee _{2r}$ be the middle term of a general $\Oo_{\PP^3}$-extension
\begin{equation}\label{eqfff1}
0 \to \Ee _{2r-1}\to \Ee _{2r}\stackrel{u}{\to} \Oo _H\to 0.
\end{equation}
By its inductive definition each $\Ee _{2r}$ is an $\Oo_X$-Ulrich sheaf with a filtration $0=\Ee_0\subset \cdots \subset \Ee _{2t}$ with $\Ee _i/\Ee _{i-1}\cong \Oo _H$ for all $i$. We will show that $\Ee_{2r}$ is indecomposable.

\quad \emph{Claim 1:} $\mathrm{ext}_{\PP^3} ^2(\Oo_H, \Ee _{2r-1})=0$ for all $r\in {\left(\frac{1}{2}\right)}{\ZZ}_{\geq 1}$.

\quad \emph{Proof of Claim 1:} for $r=1$, it is well-known that $\mathrm{ext}_{\PP^3} ^2(\Oo_H, \Oo_H)=0$. Applying the functor $\Hom_X(\Oo_H, -)$ to (\ref{eqfff1}) we obtain the strain
$$\Ext_{\PP^3} ^2(\Oo _H,\Ee _{2r-1}) \to \Ext_{\PP^3} ^2(\Oo _H,\Ee _{2r-1}) \to \Ext_{\PP^3} ^2(\Oo_H, \Oo_H).$$

\noindent We get the Claim 1 by induction on $r\in {\left(\frac{1}{2}\right)}{\ZZ}_{\geq 1}$. \qed

\quad \emph{Claim 2:} $\mathrm{ext} ^1_{\PP^3} (\Oo _H,\Ee _{2r-1}) \ne 0$ for all $r\in {\left(\frac{1}{2}\right)}{\ZZ}_{\geq 1}$.

\quad \emph{Proof of Claim 2:} Claim is true for $r=1$ by case $m=1$ of Lemma \ref{ttt}. So we assume $r>1$ and apply the functor $\Hom_{\PP^3} (\Oo _H,-)$ to (\ref{eqfff1}) to get the exact sequence
$$\Ext_{\PP^3} ^1(\Oo _H,\Ee _{2r}) \to \Ext_{\PP^3} ^1(\Oo _H,\Oo _H) \to \Ext_{\PP^3} ^2(\Oo_H, \Ee _{2r-1})=0.$$

\noindent Then we can conclude by Claim 1. \qed

\quad \emph{Claim 3:} For each positive $r\in {\left(\frac{1}{2}\right)}{\ZZ}$, $ \hom_{\PP^3} (\Oo _H,\Ee _{2r}) =1$.
	
\quad \emph{Proof of Claim 3:} We use induction on the integer $2r$, the case $2r=1$ being obvious and the case $2r=2$ being true, because $\Oo _H^{\oplus 2} \not\cong \Ii _L(1)$ for any line $L\subset H$.
Now assume $2r \ge 3$ and that $ \hom_{\PP^3} (\Oo _H,\Ee _{2r}) >1$. Since $ \hom_{\PP^3}(\Oo _H,\Ee _{2r-1}) =1$ by the inductive assumption, there is an inclusion $j: \Oo _H\rightarrow \Ee _{2r}$ such that $u\circ j : \Oo _H\rightarrow \Oo _H$ is the identity map. Hence (\ref{eqfff1}) splits, contradicting Claim 2. \qed

To conclude the proof of Theorem \ref{tthhmm} it is sufficient to prove that $\Ee _{2r}$ is indecomposable. Assume $\Ee _{2r}\cong \Ff_1\oplus \Ff_2$ with each $\Ff_i$ nontrivial. Note that each $\Ff_i$ is aCM and initialized with $h^0(\Ee _{2r})=h^0(\Ff_1)+h^0(\Ff_2)$. So each $\Ff_i$ is Ulrich and by Corollary \ref{ff2} it has a filtration starting with $\Oo _H$. Thus we get $\hom_X (\Oo _H,\Ee _{2r}) = \hom_X (\Oo _H,\Ff_1)+\hom _X(\Oo _H,\Ff_2 )\ge 2$, contradicting Claim 3.
\end{proof}

\bibliographystyle{amsplain}
\providecommand{\bysame}{\leavevmode\hbox to3em{\hrulefill}\thinspace}
\providecommand{\MR}{\relax\ifhmode\unskip\space\fi MR }
\providecommand{\MRhref}[2]{%
  \href{http://www.ams.org/mathscinet-getitem?mr=#1}{#2}
}
\providecommand{\href}[2]{#2}

\end{document}